\documentclass{article}
\usepackage[a4paper, total={6.5in, 10in}]{geometry}

\usepackage[utf8]{inputenc}
\usepackage{authblk}
\usepackage{float}
\usepackage{color}
\usepackage{mathtools}
\usepackage{amsmath,amssymb,hyperref, amsthm}
\usepackage{graphicx}
\usepackage{subcaption}
\usepackage{algorithm}%
\usepackage{algorithmic}%
\usepackage{tikz}
\usetikzlibrary{arrows}
\usepackage{xcolor}
\usepackage{cite}
\usepackage{enumerate}
\newcommand{\R}{\mathbb{R}}
\newtheorem{theorem}{Theorem}[section]
\newtheorem{corollary}[theorem]{Corollary}

\newtheorem{lemma}[theorem]{Lemma}
\newtheorem{example}[theorem]{Example}
\newtheorem{conj}[theorem]{Conjecture}

\newtheorem{definition}[theorem]{Definition}

\newtheorem{proposition}{Proposition}[section]
\newtheorem{remark}{Remark}[section]

\newtheorem*{model*}{Model}

\newcommand {\prob}{\mathbb{P}}

\newcommand{\IR}{\mbox{${\mathbb R}$}}

\newcommand{\IE}{\mbox{${\mathbb E}$}} 

\newcommand{\I}{\mbox{${\mathbb I}$}}

\newcommand{\EE}{\mbox{${\mathcal E}$}} 
\newcommand{\FF}{\mbox{${\mathcal F}$}}

\newcommand{\GG}{\mbox{${\mathcal G}$}}

\newcommand{\VV}{\mbox{${\mathcal V}$}} 
\newcommand{\PP}{\mbox{${\mathcal P}$}} 
\newcommand{\DD}{\mbox{${\mathcal D}$}}

\newcommand{\tA}{\mbox{$\widetilde{A}$}}

\newcommand{\Z}{\mbox{${{\mathbf Z}}$}}

\newcommand{\bone}{\mathbf{1}}
\newcommand{\bzero}{\mathbf{0}}

\newcommand{\bZ}{\mbox{${\mathbf Z}$}}

\title{Interacting Stochastic Systems with Agent-Based Sampling and Connection-Based Reinforcement}

\author[1]{Yogesh Dahiya%
\thanks{\texttt{ph20002@iisermohali.ac.in}}}

\author[2]{Gursharn Kaur%
\thanks{\texttt{gursharn@virginia.edu}}}

\author[1]{Neeraja Sahasrabudhe%
\thanks{\texttt{neeraja@iisermohali.ac.in}}}

\affil[1]{Indian Institute of Science Education and Research Mohali
(IISERM), India}

\affil[2]{Biocomplexity Institute, University of Virginia, USA}\date{}

\begin{document}

\maketitle

\begin{abstract}
We introduce a general class of interacting reinforced stochastic systems evolving on a network. The dynamics are governed by two distinct interaction structures: a sampling network, which determines how information is acquired, and a reinforcement network, which determines how the sampled information influences future evolution. Processes generate Bernoulli observations through preferential, de-preferential, or mixed sampling, while observations are transmitted through an edge-dependent affine reinforcement mechanism.

We formulate the dynamics as a stochastic approximation scheme and study their asymptotic behaviour. We establish sufficient conditions for almost sure convergence to a deterministic limit and, when the sampling and reinforcement graphs coincide, derive graph-theoretic criteria involving odd cycles, mixed-sampling vertices, and subcritical reinforcement. Under additional structural assumptions, we establish almost sure convergence and partial synchronization in a random-limit regime. In the deterministic regime, we also prove a law of the iterated logarithm and a central limit theorem under suitable spectral and regularity conditions. Finally, we show how the framework encompasses interacting urn models with edge-dependent reinforcement, urns with memory, and random walks with memory.
\end{abstract}

\section{Introduction}
Interacting stochastic systems arise naturally when the evolution of several random processes is coupled through a network. In many such systems, each component evolves depending on random observations of the current and/or past states of the other components. Reinforced stochastic processes provide an important framework for studying such systems. A central question is how local interaction rules, combined with reinforcement, give rise to macroscopic behaviour. Phenomena such as consensus formation (or synchronization), clustering, and polarization arise naturally in these systems. Understanding these effects requires analysing the interplay between reinforcement and the underlying network structure. The manner in which influence propagates across a network, together with the way past interactions are accumulated, plays a fundamental role in determining the long-term behaviour of the system.

In this direction, Aletti, Crimaldi and Ghiglietti \cite{AoP2017,SPA2024networks} consider networks of reinforced stochastic processes in which the probability of a new observation at a vertex is determined by a weighted combination of the current states of the processes in the network. Their work establishes almost sure synchronization and corresponding central limit theorems under appropriate conditions, and subsequently provides a first-order asymptotic classification in terms of the interaction network, including regimes of complete synchronization, convergence without synchronization, and synchronization within cyclic classes with asymptotic periodicity. Our framework differs in two related aspects. First, the mechanism governing how a process samples the states of other processes is separated from the mechanism by which the resulting observation is reinforced: sampling and reinforcement may be governed by distinct directed networks, and reinforcement is allowed to depend on the individual edge through which an observation is transmitted. Second, even when the sampling and reinforcement networks coincide, the resulting drift is not determined simply by a single adjacency or weighted-adjacency matrix. Rather, the sampling and reinforcement mechanisms combine to produce a drift involving their product, together with the node-dependent sampling parameters. This leads to a richer dependence of the asymptotic behaviour on the network structure, including effects of graph periodicity and the interaction between sampling and reinforcement.

We develop a framework for interacting reinforced stochastic systems in which sampling and reinforcement are governed by two distinct networks. The first network describes sampling interactions and determines how the current state variables are combined to generate observations. The second network describes reinforcement interactions and determines how those observations contribute to the evolution of the state variables. Unlike classical models, reinforcement is assigned through individual edges rather than vertices, allowing different interactions to induce different reinforcement effects. Sampling behaviour ranges from preferential to de-preferential, with mixed-sampling rules interpolating between the two extremes. While preferential and de-preferential sampling describe whether a process tends to align with or oppose the information available in its neighbourhood, reinforcement is further classified as critical or subcritical according to the total strength of the incoming feedback. In the critical case, the incoming reinforcement can sustain the full strength of the interaction, whereas in the subcritical case the total feedback is strictly weaker. This distinction plays an important role in determining whether randomness can persist asymptotically. Preferential and de-preferential sampling have previously been studied in interacting urn models; see, for example, \cite{bandyopadhyay2018linear,kaur2025interacting}.

We show that the asymptotic behaviour is jointly governed by the topology of the sampling and reinforcement networks and by the strength and type of reinforcement. Using the stochastic approximation representation, we first study the deterministic-limit regime through the spectral properties of the associated drift matrix. When the sampling and reinforcement graphs coincide, we obtain graph-theoretic conditions under which the drift matrix is invertible and identify the roles played by odd cycles, graph periodicity, mixed-sampling vertices, and subcritical reinforcement. We also study a random-limit regime and, under suitable assumptions, establish almost sure convergence with partial synchronization, whereby different subsets of vertices synchronize to common random limits.

Our analysis goes beyond first-order convergence to study fluctuations in the deterministic regime. We prove a law of the iterated logarithm and a central limit theorem under appropriate regularity conditions. The limiting fluctuations depend explicitly on the reinforcement mechanism and the spectral structure of the drift operator.

The framework encompasses several classes of interacting stochastic models. Urn models provide a natural probabilistic setting for studying reinforcement mechanisms, and interacting urn models have been studied extensively in \cite{MR3217785, crimaldi2016fluctuation, aletti2017interacting, sahasrabudhe2016synchronization, kaur2023interacting, crimaldi2023interacting, kaur2025interacting, MR3269167, MR3346459, MR2079914}. We show how \emph{urns with memory} can be represented by augmenting the network with auxiliary memory vertices. This construction naturally requires distinguishing between sampling and reinforcement graphs and illustrates why reinforcement must be specified at the level of edges rather than vertices. A second application concerns random walks with memory, providing a connection between the general framework and reinforced random walk models, including elephant random walks (ERWs).

The remainder of the paper is organized as follows. In Section~\ref{sec:model}, we introduce the model and derive its stochastic approximation representation. Section~\ref{sec:convergence_deter} studies the deterministic-limit regime and establishes graph-theoretic criteria for convergence, while Section~\ref{sec:LILCLT} presents the corresponding fluctuation results. Section~\ref{sec:convergence_random} investigates the random-limit regime and establishes the structural properties of the limiting random vector. The proofs of the results in Section~\ref{sec:LILCLT} and \ref{sec:convergence_random} are presented in Section~\ref{sec:proofs}. Section~\ref{sec:App} discusses several applications, including interacting urn processes, memory-based reinforcement models, and random walks. Finally, Section~\ref{sec:conc} discusses extensions and open problems. Appendix~\ref{sec:extra} contains supplementary results, while Appendix~\ref{app:appendix} collects the auxiliary results used in the proofs.

\section{Interacting Reinforced Stochastic Systems: Model Description}\label{sec:model}

Consider a finite collection of interacting reinforced stochastic processes indexed by the vertex set $\VV=[N] = \{1,\ldots,N\}$. The
evolution of the system is governed by two directed weighted graphs
\[
\mathcal G^S=(\VV,\mathcal E^S),\quad
\mathcal G^R=(\VV,\mathcal E^R),
\]
which describe the sampling and reinforcement mechanisms, respectively. 
Associated with each vertex $i\in\VV$ is a stochastic process $\{Z_t(i)\}_{t\ge0}$ taking values in $[0,1]$. The state of the system at time $t$ is represented by
\[
\Z_t=\left(Z_t(1), \ldots, Z_t(N) \right).
\]
Given the current state, the dynamics evolve recursively through two successive operations. First, a collection of random observations is generated according to the sampling mechanism. These observations are then used by the reinforcement mechanism to produce the next state vector $\Z_{t+1}$.


\paragraph{Sampling dynamics} 
Let $W=(w_{kj})_{k,j\in \VV}$ be a non-negative matrix of sampling weights, where \(w_{kj}>0\)  if and only if \((k,j)\in\mathcal E^S\). Since only the relative magnitudes of the sampling weights are relevant, any non-negative matrix  $W$ may be normalized column-wise whenever the column sum is positive. This normalization does not alter the dynamics and therefore, throughout the paper, we assume without loss of generality that $W$ is column-stochastic: $\sum_{k=1}^N w_{kj}=1$, for every $j\in \VV$. For each vertex $j\in \VV$, define the \emph{sampling neighborhood} of vertex $j$ by $N_W(j)=\{k\in \VV:w_{kj}>0\}$
. Associated with each vertex $j\in\VV$ is a sampling parameter $q_j\in[0,1]$.  Process $j$ samples preferentially with probability $q_j$, and de-preferentially with probability $1-q_j$. 

Given the filtration  $\FF_t = \sigma\!\left( Z_s(i): 0\le s\le t,\; i\in\VV \right)$, the sampling indicators $\{\I_{t+1}(i)\}_{i\in[N]}$ are independent Bernoulli random variables satisfying:
\begin{equation} \label{eq:sampling_indicator}
    \I_{t+1}(i) =
\begin{cases}
1 & \text{w.p.} \quad  q_i \left( \sum_{k \in \VV} w_{ki} Z_t(k) \right) + (1 - q_i)\left(1 - \sum_{k \in \VV} w_{ki} Z_t(k)\right) \\
0 & \text{otherwise}.
\end{cases}
\end{equation}
Note that when $q_i = 1/2$, the sampling indicators are Bernoulli(1/2) and independent of all the other processes in the system. As in \cite{kaur2025interacting}, we call the sampling at node $i$  \textbf{\textit{preferential}} if $q_i=1$; \textbf{\textit{de-preferential}} if $q_i=0$; \textbf{\textit{mixed-sampling}} if $q_i\in(0,1)$. We denote the set of nodes with preferential and de-preferential sampling by $\mathcal{P}$ and $\mathcal{D}$ respectively. 

\paragraph{Edge-Based Reinforcement} 
Let $\GG^R=(\VV,\EE^R) $  be the reinforcement graph and define $N(i) = \{j\in\VV:(j,i)\in\EE^R\}.$ To every reinforcement edge $(j,i)\in\mathcal{E}^R$, we associate a reinforcement map $R_{ji}:[0,1]\to\mathbb{R},$
which specifies the contribution of the observation at vertex $j$ to the evolution of the state variable $Z_t(i)$. The collection $\mathcal{R}=\{R_{ji}:(j,i)\in\mathcal{E}^R\}$ is called the \emph{reinforcement mechanism} on $\GG^R$. Throughout this paper, we restrict our attention to \emph{affine reinforcement mechanisms}, that is,
\begin{equation} \label{eq:reinfocement_affine}
    R_{ji}(x)=\mu_{ji}x+c_{ji},
\end{equation}
where $\mu_{ji}\in[-1,1]$ and $c_{ji}\geq0$.
For pairs $(j,i)\notin\mathcal E^R$, we set
$\mu_{ji}=c_{ji}=0$ and $R_{ji}\equiv0$.
Thus sums involving these reinforcement coefficients or maps
may be taken over all vertices, with non-edges contributing zero.
We further assume that, for every vertex $i\in\VV$,\begin{equation} \label{reinforcement_assumption}
\sum_{j\in N(i)}c_{ji}
+ \sum_{\substack{j\in N(i)\\ \mu_{ji}<0}}\mu_{ji}
\geq0,
\qquad \sum_{j\in N(i)}c_{ji}
+ \sum_{\substack{j\in N(i)\\ \mu_{ji}>0}}\mu_{ji}
\leq 1.
\end{equation}
These conditions ensure that the total reinforcement received by a vertex remains uniformly bounded and that the state variables continue to evolve within the interval $[0,1]^N$.  

\ 

\noindent For each vertex \(i\in\VV\), define \begin{equation}
    \rho_i=\sum_{j\in N(i)}|\mu_{ji}|.
\end{equation}
We call $i$ \emph{subcritical} if \(\rho_i<1\), \emph{critical} if $\rho_i=1$, and \emph{positive-critical} if $\rho_i=1$ and $\mu_{ji}>0$ for every $j\in N(i)$.


\paragraph{Dynamics and Stochastic Approximation Representation} 
The state variables evolve according to the recursive scheme
\begin{equation}\label{eq:gen_recur}
Z_{t+1}(i) = Z_t(i)\left(1-\frac{1}{t+1}\right) +\frac{1}{t+1} \sum_{j\in N(i)} R_{ji}\!\big(\I_{t+1}(j)\big). 
\end{equation}
More generally, we consider the perturbed version
\begin{equation}\label{eq:gen_recur_perturb}
Z_{t+1}(i)
=
Z_t(i)+\frac1{t+1}\left[
\sum_{j\in N(i)}R_{ji}\bigl(\I_{t+1}(j)\bigr)
-Z_t(i)+\epsilon_{t+1}(i)
\right]. 
\end{equation}
Here $\epsilon_t$ is $\mathcal F_t$-measurable, uniformly bounded such that $Z_t(i) \in [0, 1]$ for all $i \in \VV$ and $t \geq 0$, and $\epsilon_t\to0$ almost surely.
 Conditioned on $\FF_t$, the sampling indicators at time $t+1$ are independent of each other. Substituting the affine reinforcement maps
\eqref{eq:reinfocement_affine} into \eqref{eq:gen_recur_perturb}, and decomposing the observations
into their conditional means and centered fluctuations,
gives
\begin{align}
Z_{t+1}(i)
&=
Z_t(i)+\frac1{t+1}\Bigg[
\sum_{j\in N(i)}
\left(\mu_{ji}\I_{t+1}(j)+c_{ji}\right)
-Z_t(i)+\epsilon_{t+1}(i)
\Bigg]
\label{eq:main_recur}\\
&=
Z_t(i)+\frac1{t+1}\Bigg[
\sum_{j\in N(i)}
\bigg\{
\mu_{ji}\bigg(
(2q_j-1)\sum_{k\in\VV}w_{kj}Z_t(k)+(1-q_j)
\bigg)+c_{ji}
\bigg\}
-Z_t(i)+\Delta M_{t+1}(i)+\epsilon_{t+1}(i)
\Bigg],
\nonumber
\end{align}
where
$
\Delta M_{t+1}(i)
=
\sum_{j\in N(i)}\mu_{ji}
\left(
\I_{t+1}(j)
-\mathbb E\big[\I_{t+1}(j)\mid\mathcal F_t \big]
\right)$
 is a martingale difference sequence. 
In vector form this can be written as follows. 
 \begin{equation}\label{eq:main_recur_vec}
    \Z_{t+1}=\Z_t+\frac{1}{t+1}\left[\Z_t(WQ\tilde{A}-I)+C + \Delta M_{t+1} +\epsilon_{t+1}\right],
\end{equation}
 where all vectors are understood as row vectors, $\Z_{t}  = (Z_t(i))_{i\in[N]}$, $\Delta M_{t}=(\Delta M_{t}(i))_{i\in[N]}$, $\epsilon_{t}=(\epsilon_{t}(i))_{i\in[N]}$, $Q=\operatorname{diag}(2q_1-1,\ldots,2q_N-1)$, and $\widetilde A=(\mu_{ji}\I_{\{j\in N(i)\}})_{j,i\in[N]}$. Finally, $C=(C_i)_{i\in[N]}$, where, for $i\in[N]$, $C_i=\sum_{j\in N(i)}\bigl(\mu_{ji}(1-q_j)+c_{ji} \bigr)$. 
 Equation \eqref{eq:main_recur_vec} is a stochastic approximation recursion whose drift is determined by the matrix $H=WQ\widetilde A-I.$ Note that the drift is not simply a linear function of the adjacency matrix of the underlying interaction graph; it is determined by a composition of sampling and reinforcement operators.
 
\paragraph{Overview and Motivation} 
One of the applications of the interacting stochastic process model of the type in \eqref{eq:main_recur_vec} is to opinion dynamics where the $Z_t(i) \in  [0, 1]$ represents the inclination of the individual $i$ at time $t$. A binary opinion can be obtained by thresholding as follows $O_t(i) = Sign \left( Z_t(i) - \frac{1}{2} \right)$ as in \cite{kaur2023interacting}. Another motivation is the study of binary social learning models on networks. Consider a population of $N$ agents connected by a directed social network. There is an unknown binary state of nature
$\Theta\in\{0,1\}$ representing, for instance, whether a circulating claim is true or false. At time $t$, agent $i$ maintains a confidence level
$Z_t(i)\in[0,1]$, interpreted as its belief that the claim is true. At each time step, every agent first consults its neighbours before deciding whether to publicly endorse the claim. Let $L_t(i)=\sum_{k\in\VV}w_{ki}Z_t(k)$ denote the weighted social evidence available to agent $i$. Suppose the endorsement indicator $I_{t+1}(i)\in\{0,1\}$
satisfies
\[
\mathbb{P}\!\left(I_{t+1}(i)=1\mid\mathcal{F}_t\right)
=
q_iL_t(i)
+
(1-q_i)\bigl(1-L_t(i)\bigr).
\]
The parameter $q_i\in[0,1]$ captures the behavioural tendency of agent $i$, with larger values corresponding to stronger conformity to the opinions of its neighbours and smaller values corresponding to increasingly contrarian behaviour. Unlike classical opinion dynamics models, here different neighbours need not have the same influence on an agent's future belief. Let $\mu_{ji}$ denote the influence exerted by agent $j$ on agent $i$, and let $c_{ji}$ represent a baseline reinforcement associated with this interaction. The total reinforcement received by agent $i$ at time $t+1$ is then given by
\[
S_{t+1}(i)
=
\sum_{j\in N(i)}
\left(\mu_{ji}I_{t+1}(j)+c_{ji}\right).
\]
Finally, agent $i$ updates its belief by forming the running empirical average
\[
Z_{t+1}(i)
=
\frac{t}{t+1}Z_t(i)
+
\frac{1}{t+1}S_{t+1}(i).
\]
Substituting the expression for $S_{t+1}(i)$ yields the reinforcement recursion \eqref{eq:gen_recur}. Thus, binary social learning with heterogeneous trust and confirmation bias is naturally described by the interacting reinforced stochastic framework introduced in this paper. Several other stochastic systems with correlated sub-parts behave in the same way. In section~\ref{sec:App} we discuss in detail the application of our model of interacting processes and the corresponding asymptotic results to certain interacting urn processes and interacting random walks with memory on graphs.

The framework extends existing interacting reinforced systems in two fundamental directions. First, reinforcement is assigned to edges rather than vertices, allowing different interactions to induce different reinforcement effects. Second, sampling and reinforcement are represented by distinct interaction networks, thereby decoupling information acquisition from reinforcement. As a consequence, the dynamics depend jointly on the sampling graph, the reinforcement graph, and the associated edge parameters, yielding a substantially richer class of interacting stochastic systems than those arising from a single interaction matrix. The examples discussed above and those presented in section~\ref{sec:App} illustrate how a variety of reinforced processes can be naturally embedded within this framework.

\section{Deterministic Limit Regime}\label{sec:convergence_deter}
\subsection{Convergence via Stochastic Approximation}
\paragraph{Stochastic Approximation Representation}  The recursion \eqref{eq:main_recur_vec} is of the standard stochastic approximation form, with the sampling and reinforcement mechanisms jointly determining the mean-field dynamics. Consequently, the asymptotic behaviour of the system can be analysed using stochastic approximation techniques under suitable assumptions on the drift, martingale noise and remainder terms. The asymptotic behaviour of the process is therefore determined by the matrix \[ H=WQ\widetilde A-I, \] whose structure reflects the combined effects of the sampling and reinforcement mechanisms.
\begin{theorem}[Almost sure convergence of $\mathbf Z_t$]
\label{thm:deterministic_limit}
Let $H=WQ\widetilde A-I$. If $H$ is invertible, then, as
$t\to\infty$, the process $\mathbf Z_t$ converges
almost surely to the deterministic limit
\begin{equation} \label{deter_limit}
\bZ^*=-C(WQ\widetilde A-I)^{-1}.
\end{equation}
\end{theorem}

\begin{proof}
The reinforcement assumptions imply
$\sum_j|\widetilde A_{ji}|\leq1$ for every $i$.
Hence, in the induced matrix $1$-norm,
\[
\|WQ\widetilde A\|_1
\leq
\|W\|_1\|Q\|_1\|\widetilde A\|_1
\leq1.
\]
Every eigenvalue $\zeta$ of $WQ\widetilde A$ therefore
satisfies $|\zeta|\leq1$. Since
$H=WQ\widetilde A-I$ is invertible, $\zeta\ne1$,
and consequently $\operatorname{Re}(\zeta-1)<0$.
Thus all eigenvalues of $H$ have strictly negative
real parts. The associated ODE
\[
\dot z=zH+C
\]
has the unique globally asymptotically stable equilibrium
$-CH^{-1}$. In \eqref{eq:main_recur_vec}, the state
variables and martingale differences are bounded,
the step sizes are $1/(t+1)$, and the remainder tends
to zero almost surely. Lemma~\ref{lem:sas} therefore
gives the claimed convergence.
The same argument applies when the recursion is
initialized at any fixed finite time.
\end{proof}


Theorem~\ref{thm:deterministic_limit} reduces the problem
of deterministic convergence to the invertibility of
the drift matrix $H$, without requiring the sampling
and reinforcement graphs to coincide.
To obtain graph-theoretic sufficient conditions for
invertibility, we impose the following additional
standing assumption for the remainder of this section.
\begin{itemize}
\item[{\bf (A1)}]
The sampling and reinforcement graphs have the same
underlying directed edge set, denoted by $\mathcal E$.
We write $\mathcal G=(\VV,\mathcal E)$ for the resulting
graph and assume that $\mathcal G$ is strongly connected.
\end{itemize}
Under this standing assumption, we obtain graph-theoretic
conditions ensuring that $H$ is invertible.

We introduce some notation that will be used throughout. For vertices $i,j\in\VV$, we write $i\to j$ if $(i,j)\in\mathcal E$. We write $i\rightsquigarrow j$ if there exists a directed path from $i$ to $j$ in $\mathcal G$, that is, if there exist vertices $i_1,\ldots,i_k\in\VV$ such that $i\to i_1\to\cdots\to i_k\to j.$ Given a directed path $\ell$ from $i$ to $j$, let $d_\ell(i\rightsquigarrow j)$ denote its length, namely the number of edges in the path. For a vertex $a\in\VV$ and an integer $n\geq 1$, define \[ N_a^{(n)} = \Big\{ v\in\VV: d_\ell(v\rightsquigarrow a)=n \text{ for some directed path } \ell \Big\}. \] Thus, $N_a^{(n)}$ consists of all vertices that can reach $a$ through a directed path of length $n$. In the definition of $N_a^{(n)}$, repeated vertices are allowed; thus the directed paths considered there are directed walks.

\begin{theorem}[Invertibility via mixed sampling or subcritical reinforcement] \label{thm:invertibility_1}
    Assume that, for every critical vertex $i$,
$\mu_{ji}\ne0$ for every $j\in N(i)$.
 Then $H$ is invertible if either
of the following conditions holds:
    \begin{enumerate}[(i)]
        \item there exists a mixed-sampling vertex and either $\mathcal G$ contains an odd cycle or two mixed-sampling vertices are connected by a directed path of odd length;
        \item there exists a subcritical vertex and either $\mathcal G$ contains an odd cycle or two subcritical vertices are connected by a directed path of odd length.
    \end{enumerate}
Consequently, by Theorem~\ref{thm:deterministic_limit}, \[ \Z_t \xrightarrow{\mathrm{a.s.}} -CH^{-1}. \] 
\end{theorem}
\begin{proof}
    Suppose, for contradiction, that $H$ is singular. Since $H$ is a real matrix, there exists a nonzero real row vector $v=(v_1,\ldots,v_N) $  such that $vH=0$, i.e. $v=vWQ\widetilde A.$ Hence, for each $k\in\VV$, 
    \begin{equation} \label{eq:nullvector} 
    v_k = \sum_{r\in\VV} (2q_r-1)\mu_{rk} \sum_{j=1}^{N} w_{jr}v_j . 
    \end{equation} 
    Choose $l\in\VV$ such that $|v_l|=\max_{i\in\VV}|v_i|,$ and define $\widetilde v=\frac{v}{|v_l|}.$ Then 
    \[ |\widetilde v_l|=1, \qquad |\widetilde v_i|\le 1, \quad i\in\VV . \]
    Applying \eqref{eq:nullvector} with $k=l$ gives
    \[ 1 = |\widetilde v_l| = \left| \sum_{r\in\VV} (2q_r-1)\mu_{rl} \sum_{j=1}^{N} w_{jr}\widetilde v_j \right|. \]
    Using the triangle inequality, $1 \le \sum_{r\in\VV} |2q_r-1|\,|\mu_{rl}| \sum_{j=1}^{N} w_{jr}|\widetilde v_j|.$
    Since $W$ is column-stochastic and $|\widetilde v_j|\le 1$ for all $j$, 
    \[ 1 \le \sum_{r\in\VV} |2q_r-1|\,|\mu_{rl}| \le \sum_{r\in\VV}|\mu_{rl}| \le 1, 
    \] where the last inequality follows from \eqref{reinforcement_assumption}. Therefore equality must hold throughout. Consequently,
    \begin{enumerate}
    \item every in-neighbour $r$ of $l$ satisfies $|2q_r-1|=1$, i.e. $r$ is not mixed-sampling;
    \item the vertex $l$ is critical, i.e. $\sum_{r\in\VV}|\mu_{rl}|=1$; 
    \item $\sum_{j=1}^{N}w_{jr}|\widetilde v_j|=1$ for every $r\to l$. 
    \end{enumerate}
    Since the weights in each column of $W$ are nonnegative and sum to one, the third condition implies that $|\widetilde v_j|=1$ for every $j$ with $w_{jr}>0$ and every $r\to l$. Hence 
    \[ |\widetilde v_j|=1, \qquad j\in N_l^{(2)}. \]
    Repeating the same argument at any vertex $x$ satisfying $|\widetilde v_x|=1$, we obtain inductively that 
    \[ |\widetilde v_j|=1, \qquad j\in N_l^{(2m)}, \quad m\ge 1. \]
    Moreover, for every $m\ge0$, 
    \begin{itemize} 
    \item every vertex in $N_l^{(2m+1)}$ is not \textit{mixed-sampling},
    \item every vertex in $N_l^{(2m)}$ is \textit{critical}.
    \end{itemize}
    Thus, the existence of a nonzero null vector forces all vertices in odd layers to satisfy $q_j\in\{0,1\}$ and all vertices in even layers to be critical.

    We now derive a contradiction. Suppose condition {\rm (i)} holds. First assume that $\mathcal G$ contains an odd cycle and that there exists a vertex $r^*$ with $|2q_{r^*}-1|<1$.  Since $\mathcal G$ is strongly connected, there exists a directed path from $r^*$ to $l$. Combining this path with the odd cycle, if necessary, yields a directed path from $r^*$ to $l$ of odd length. Hence $r^*\in N_l^{(2m+1)}$ for some $m\ge0$, contradicting the fact that every vertex in an odd layer must satisfy $|2q_j-1|=1$.

    Now, suppose there exist two mixed-sampling vertices $r^*$ and $r'$ connected by a directed path of odd length, i.e.  $r^*\rightsquigarrow r'$ (or $r'\rightsquigarrow r^*$) via a path of odd length. Then relative to $l$, at least one of the vertices $r^*$ and $r'$ must belong to an odd layer, giving the same contradiction. 

Next assume that condition {\rm (ii)} holds. The argument is analogous. If $\mathcal G$ contains an odd cycle and there exists a subcritical vertex $v$, then there exists an even-length directed walk from $v$ to $l$. Thus $v\in N_l^{(2m)}$ for some $m\ge0$, contradicting the fact that every even-layer vertex is critical.
    
Similarly, if two subcritical vertices are connected by a directed path of odd length, then relative to $l$ at least one of them must belong to an even layer, again contradicting the requirement that every even-layer vertex be critical.  Therefore $H$ is invertible. The convergence statement follows from Theorem~\ref{thm:deterministic_limit}. 
\end{proof}

\subsection{A Graph Partition Characterization}
Throughout this subsection, we continue to assume {\bf (A1)}.
We now address the regime in which every vertex is positive-critical. The invertibility is determined by parity structure of the interaction graph. To characterize this structure, we introduce a graph partition algorithm (GPA), which generalizes the partition procedure proposed in \cite{kaur2025interacting}. 
We show that the invertibility of the matrix $WQ\tilde{A}-I$ is closely related to whether the given graph admits a partition via that algorithm. This builds on the algorithm proposed in \cite{kaur2025interacting}, the present formulation applies to the broader class of interacting reinforced stochastic systems with a more general drift matrix. 
Before we describe the algorithm, for every \(j\in\VV\), define an auxiliary
graph \(G_j=(\VV_j,\mathcal E_j)\), where $\VV_j=\bigcup_{k\geq 0}N_j^{(2k)}$.
For \(v_1,v_2\in\VV_j\), \((v_1,v_2)\in\mathcal E_j\) if there exists
\(w\in\VV\) such that \(v_1\to w\to v_2\) in the original graph.

\begin{figure}[h]
\centering
\begin{tikzpicture}[
    main/.style = {draw, circle, minimum size=1cm, font=\large, thick, text=white},
    midpoint/.style={circle, draw, minimum size=0.5cm, inner sep=0pt, thick},
    node distance=4cm
]
    \node[main, fill=red!75] (D1) at (0,4) {$D_1$}; 
    \node[main, fill=orange] (D2) at (4,4) {$D_2$};
    \node[main, fill=blue!70] (P1) at (0,0) {$P_1$}; 
    \node[main, fill=green!60!black] (P2) at (4,0) {$P_2$}; 

    \node[midpoint] (M_P1D2) at (2.5, 1.5) {}; 
    \draw[->, thick] (P1) to [bend left=20] (M_P1D2);
    \draw[->, thick] (M_P1D2) to [bend left=20] (D2);

    \node[midpoint] (M_P2D1) at (1.5, 2.5) {}; 
    \draw[->, thick] (P2) to [bend left=20] (M_P2D1);
    \draw[->, thick] (M_P2D1) to [bend left=20] (D1);

    \node[midpoint] (M_P1L) at (-1.5, 0) {};
    \draw[->, thick] (P1) to [bend right=45] (M_P1L);
    \draw[->, thick] (M_P1L) to [bend right=45] (P1);

    \node[midpoint] (M_P2L) at (5.5, 0) {};
    \draw[->, thick] (P2) to [bend left=45] (M_P2L);
    \draw[->, thick] (M_P2L) to [bend left=45] (P2);

    \node[midpoint] (M_D2D1) at (2, 4.7) {};
    \draw[->, thick] (D2) to [bend right=20] (M_D2D1);
    \draw[->, thick] (M_D2D1) to [bend right=20] (D1);

    \node[midpoint] (M_D1D2) at (2, 3.3) {};
    \draw[->, thick] (D1) to [bend right=20] (M_D1D2);
    \draw[->, thick] (M_D1D2) to [bend right=20] (D2);

    \node[midpoint] (M_D1P1) at (0,2) {};
    \draw[->, thick] (D1) -- (M_D1P1);
    \draw[->, thick] (M_D1P1) -- (P1);

    \node[midpoint] (M_D2P2) at (4,2) {};
    \draw[->, thick] (D2) -- (M_D2P2);
    \draw[->, thick] (M_D2P2) -- (P2);

\end{tikzpicture}
\caption[Two-step exploration rules in the GPA.]{The figure describes the two-step exploration rules used in Algorithm~\ref{alg:graph_partition}. The arrows indicate how vertices reached through two-step paths are assigned among the four sets $P_1,P_2,D_1,D_2$, depending on whether they are preferential or de-preferential.} 
\label{fig:gpa_partition}
\end{figure}
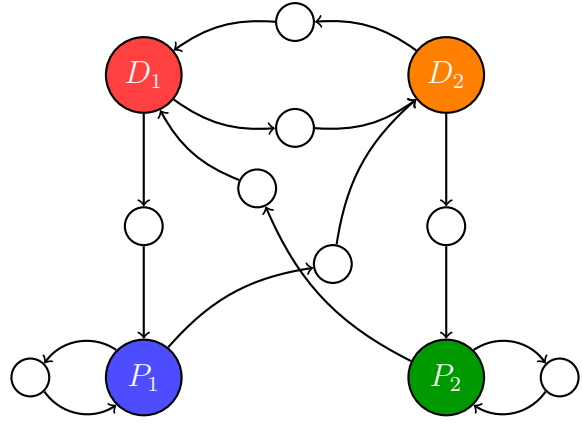

\begin{algorithm}[H]
 \caption{Graph Exploration Process} \label{alg:graph_partition}
 \textbf{Input}: A directed graph $\GG(V, E)$, sets of preferential nodes  $\PP$ and de-preferential nodes $\DD$ and a node $j \in \PP \cup \DD$. \\
\mbox{\textbf{Output}: Whether starting from $j$, $\GG$ admits a partition or not.}
 \begin{algorithmic}[1] 
\IF {$j\in \PP$}
\STATE Initialize $P_1 \gets \{j\}$, $P_2 = D_1=D_2 = \emptyset$.
\ELSE 
\STATE Initialize $D_1 \gets \{j\}$, $P_1 = P_2 =D_2 = \emptyset$.
\ENDIF
 \WHILE{$(P_1\cup P_2\cup D_1\cup D_2)  \subsetneq V$}
 \STATE  $P_1 \gets P_1 \cup  \left( \cup_{j\in P_1} N^{(2)}_j  \cap \PP \right)$ and  $D_1 \gets D_1 \cup  \left( \cup_{j\in P_1} N^{(2)}_j  \cap \DD \right)$
  \STATE  $P_2 \gets P_2 \cup  \left( \cup_{j\in D_1} N^{(2)}_j  \cap \PP \right)$ and $D_2 \gets D_2 \cup  \left( \cup_{j\in D_1} N^{(2)}_j  \cap \DD \right)$
  \STATE  $P_1 \gets P_1 \cup  \left( \cup_{j\in D_2} N^{(2)}_j  \cap \PP \right)$ and  $D_1 \gets D_1 \cup  \left( \cup_{j\in D_2} N^{(2)}_j  \cap \DD \right)$
  \STATE  $P_2 \gets P_2 \cup  \left( \cup_{j\in P_2} N^{(2)}_j  \cap \PP \right)$ and  $D_2 \gets D_2 \cup  \left( \cup_{j\in P_2} N^{(2)}_j  \cap \DD \right)$
    \IF {$P_1, P_2, D_1, D_2$ are not mutually disjoint} 
 \STATE {\bf BREAK} and \textbf{return} \lq\lq$\GG_j$ does not admit a graph partition."  
 \ENDIF
 \ENDWHILE
 \STATE  Repeat Steps 8--11 once.
\IF {any node is re-assigned from $P_1$ to $P_2$ (or vice versa) or from $D_1$ to $D_2$ (or vice versa)} 

 \STATE \textbf{return} \lq\lq$\GG_j$ does not admit a graph partition"
 \ELSE
 \STATE \textbf{return} \lq\lq$\GG_j$ admits a graph partition $G^{(2)}_j(P_1, P_2, D_1, D_2)$." 
 \ENDIF
  \end{algorithmic}
 \end{algorithm}

Note that the graph partition algorithm in \cite{kaur2025interacting} is recovered by replacing the two-step neighbourhood $N_j^{(2)}$ by the one-step in-neighbourhood $N(j)$. In that setting, the outcome of the algorithm is independent of the initial vertex $j$, up to a permutation of the four sets $P_1,P_2,D_1,D_2$. In the present model, however, the algorithm explores the graph through two-step neighbourhoods, and hence the outcome may depend on the initial vertex. In particular, starting from a given vertex, the algorithm may fail to reach all preferential and de-preferential nodes, even when the underlying graph is strongly connected. For this reason, the sets $P_1,P_2,D_1,D_2$ should be viewed as depending on the initial vertex, although we suppress this dependence in the notation. Finally, if a partition is obtained, the labels $1$ and $2$ are immaterial: simultaneously exchanging $P_1$ with $P_2$ and $D_1$ with $D_2$ gives the same partition structure.

The graph partition algorithm can be thought of as a graph exploration process. Starting from a vertex $j$ at time 0, the process moves at each subsequent step to a vertex at directed distance two and assigns that vertex to one of the four classes. 
\begin{theorem} \label{thm:GPA}
Suppose that $\GG$ contains an odd cycle.
Assume further that all vertices are positive-critical
and that $q_v\in\{0,1\}$ for every $v\in\VV$.
Then $WQ\widetilde A-I$ is invertible if and only if
$\GG$ does not admit a partition under the GPA.
\end{theorem}

\begin{proof}
All vertices are positive-critical and all nodes are either preferential or de-preferential, 
we have $\sum_{j\in\VV}\mu_{jr}=1$  and $2q_r-1\in\{-1,1\}$ for each $r \in [N]$. First suppose that $\GG$ admits a partition under the GPA. Define a sign vector
$\sigma=(\sigma_i)_{i\in[N]}$ by
\[
\sigma_i=
\begin{cases}
1, & \text{if } i\in P_1\cup D_2,\\
-1, & \text{if } i\in P_2\cup D_1.
\end{cases}
\]

The pattern of the GPA partition shown in Figure~\ref{fig:gpa_partition}
implies that for every directed two-step path $i\to r\to k$,
$\sigma_k=(2q_i-1)\sigma_i$. Indeed, if $i\in\PP$, then $2q_i-1=1$, and the GPA places $i$ and $k$ in
classes carrying the same  sign. If $i\in\DD$, then $2q_i-1=-1$,
and the GPA places $i$ and $k$ in classes carrying opposite  signs. Using this, we get that for every $k\in[N]$,
\begin{eqnarray*} (\sigma Q\widetilde A W)_k &=& \sum_{r\in\VV} w_{rk} \sum_{i\in\VV}(2q_i-1)\mu_{ir}\sigma_i= \sum_{r\in\VV} w_{rk} \sum_{i\in\VV}\mu_{ir}\sigma_k\\ &=& \sigma_k\sum_{r\in\VV} w_{rk} \sum_{i\in\VV}\mu_{ir}= \sigma_k. 
\end{eqnarray*}

\noindent 
Thus, $\sigma Q\widetilde A W=\sigma$. Set $v=\sigma Q\widetilde A$. Then
$v\neq0$, since $vW=\sigma Q\widetilde A W=\sigma\neq0$. Moreover, $vWQ\widetilde A
=
\sigma Q\widetilde A WQ\widetilde A
=
\sigma Q\widetilde A
=
v$.
Therefore $v(WQ\widetilde A-I)=0$, and so $WQ\widetilde A-I$ is singular.

To prove the converse, now suppose that $WQ\widetilde A-I$ is singular. Since this is a real matrix, there exists a non-zero real row vector $v=(v_1,\ldots,v_N)$
such that $v(WQ\widetilde A-I)=0$. Let $u=vW$. Then $u\neq0$, because otherwise
$v=vWQ\widetilde A=0$, which contradicts the choice of $v$. Also,
$uQ\widetilde A W
=
vWQ\widetilde A W
=
vW
=
u$.

Choose $l\in[N]$ such that $|u_l|=\max_{i\in[N]}|u_i|$, and normalize so that
$|u_l|=1$. Then $|u_i|\leq1$ for all $i\in[N]$. From $uQ\widetilde A W=u$,
the $k$-th coordinate satisfies
\[
u_k = \sum_{r\in\VV}w_{rk}
\sum_{i\in\VV}(2q_i-1)\mu_{ir}u_i.
\]
Now, by the triangle inequality,
\begin{align*}
|u_k|
=
\left|
\sum_{r\in\VV}w_{rk}
\sum_{i\in\VV}(2q_i-1)\mu_{ir}u_i
\right| &\leq
\sum_{r\in\VV}w_{rk}
\sum_{i\in\VV}\mu_{ir}|2q_i-1|\,|u_i|\\
&\leq
\sum_{r\in\VV}w_{rk}
\sum_{i\in\VV}\mu_{ir} =1.
\end{align*}
Now suppose that $|u_k|=1$. Since the coefficients
$w_{rk}\mu_{ir}$ are non-negative and 
$\sum_{r\in\VV}w_{rk}
\sum_{i\in\VV}\mu_{ir} = 1,$
we have that, for every directed two-step path $i\to r\to k$, $|u_i|=1$, and all the terms $(2q_i-1)u_i$ have the same sign as $u_k$. Therefore, $(2q_i-1)u_i=u_k$ for every directed two-step path $i\to r\to k$.

Starting from $l$ and repeating this argument along two-step paths, we get
$|u_i|=1$ for every vertex reached by the two-step exploration from $l$. Since
$\GG$ is strongly connected and contains an odd cycle, this two-step
exploration reaches every vertex. Therefore $u_i\in\{-1,1\}$ for every
$i\in\VV$. Moreover, for every directed two-step path $i\to r\to k$, $u_k=(2q_i-1)u_i$. Multiplying $u$ by $-1$ if necessary, we may assume that $u_l=1$ when
$l\in\PP$ and $u_l=-1$ when $l\in\DD$. Define
\[
P_1=\{i\in\PP:u_i=1\},\qquad
P_2=\{i\in\PP:u_i=-1\},
\]
and
\[
D_1=\{i\in\DD:u_i=-1\},\qquad
D_2=\{i\in\DD:u_i=1\}.
\]
These four sets are mutually disjoint and their union is $\VV$. We claim that the GPA, started from $l$, assigns vertices consistent with these
four sets. This is proved by induction over the steps of the algorithm. The
initial assignment is consistent by our choice of the sign of $u_l$. Suppose
that all vertices assigned up to some stage agree with the four sets above.
If the algorithm uses a directed two-step path $i\to r\to k$ to assign a new
vertex $i$ from an already assigned vertex $k$, then $u_k=(2q_i-1)u_i$. Hence the new assignment also agrees with the four sets above. Therefore GPA never
assigns a vertex to two incompatible classes. Since the two-step exploration
reaches every vertex, $\GG$ admits a partition under the GPA. 

We have shown that $\GG$ admits a partition under the GPA if and only if
$WQ\widetilde A-I$ is singular. Equivalently, $WQ\widetilde A-I$ is
invertible if and only if $\GG$ does not admit a partition under GPA.    
\end{proof}

 Throughout the examples below, we consider directed cycles labelled as
$1\to2\to\cdots\to k\to1$, for some $k\in\mathbb N$. All indices are read
 modulo $k$. We will always assume (A1), so the sampling and reinforcement
graphs have the same directed cycle as their common edge set. Hence the only non-zero entries of $W$ are $w_{i,i+1}$, and column-stochasticity gives
$w_{i,i+1}=1$ for each $i=1,\ldots,k$. Similarly, the only non-zero entries
 of $\widetilde A$ are $\mu_{i,i+1}$. 

\begin{corollary}
   [Odd cycles]
Consider an odd cycle with $k$ vertices. Further assume that all vertices are positive-critical and either preferential or
de-preferential. Then, $WQ\widetilde A-I$ is invertible if and only if the number of
de-preferential nodes is odd.
\end{corollary}
\begin{proof}
 The proof follows from the fact that in this case $\mu_{i,i+1}=1$ for every $i$ and $\det(WQ\widetilde A-I)=\prod_{i=1}^k(2q_i-1)-1$.    
\end{proof}
Observe that the GPA gives the same condition. If a partition exists, encode its four
classes by $\sigma_i=1$ on $P_1\cup D_2$ and $\sigma_i=-1$ on $P_2\cup D_1$.
Along the cycle, the partition pattern gives $\sigma_{i+2}=(2q_i-1)\sigma_i$. Iterating once around the odd cycle gives
$\sigma_i=\left(\prod_{r=1}^k(2q_r-1)\right)\sigma_i$. Thus a GPA partition
exists if and only if the number of de-preferential vertices is even.

\medskip 
 We now consider directed even cycles. Let the cycle have even length $2k$, and set
 $\VV_1=\{1,3,\ldots,2k-1\}$ and $\VV_2=\{2,4,\ldots, 2k\}$. The two-step
exploration preserves parity: starting from a vertex in $\VV_a$, the GPA
 only explores vertices in $\VV_a$. We denote the corresponding two-step
 components by $\GG_1$ and $\GG_2$. Note that the odd-cycle GPA criterion in Theorem~\ref{thm:GPA} does not apply. Moreover, the two-step exploration no longer reaches the whole graph from a single starting vertex: it remains trapped in one of the two parity components.  

Theorem~\ref{thm:invertibility_1} already addresses some cases. On an even directed cycle, vertices of opposite parity are connected by a directed path of odd length. So if both $\GG_1$ and $\GG_2$ contain subcritical vertices, then Theorem~\ref{thm:invertibility_1} applies, and $WQ\widetilde A-I$ is invertible. 

Hence, in the cases not covered by Theorem~\ref{thm:invertibility_1} and Theorem~\ref{thm:GPA}, all subcritical vertices have the same parity, or equivalently, they belong to either $\GG_1$ or $\GG_2$. This motivates the following component-wise criterion.

 \begin{proposition}\label{prop:evencycle}
Suppose $\GG$ is a directed even cycle of length $2k$,
with $\VV_1$ and $\VV_2$ as defined above.
Suppose that, for some $a\in\{1,2\}$, every vertex in
$\VV_a$ is positive-critical, while every vertex in
$\VV_{3-a}$ is either preferential or de-preferential.
If $|\DD\cap\VV_{3-a}|$ is even, then
$WQ\widetilde A-I$ is singular.
\end{proposition}

\begin{proof}
It suffices to consider $a=1$; the other case follows
by shifting the cyclic labelling by one vertex.
Since every odd-labelled vertex is positive-critical
and has exactly one reinforcement in-neighbour,
$\mu_{2j,2j+1}=1$ for $j=1,\ldots,k$, with indices
read modulo $2k$. Moreover,
$\prod_{j=1}^{k}(2q_{2j}-1)
=(-1)^{\left|\DD\cap\VV_2 \right|}=1.$
The determinant factors as
\[
\det(WQ\widetilde A-I)
=
\left(1-\prod_{j=1}^{k}
(2q_{2j-1}-1)\mu_{2j-1,2j}\right) \times
\left(1-\prod_{j=1}^{k}
(2q_{2j}-1)\mu_{2j,2j+1}\right).
\]
The second factor is zero, proving the assertion.
\end{proof}
In general, a strongly connected graph with no odd directed cycles is a bipartite digraph. That is, there exist vertex sets $U_1$ and $U_2$ such that all directed edges from nodes in $U_1$ point to $U_2$ and all directed edges from $U_2$ point to nodes in $U_1$. Moreover, there are no directed edges from nodes in $U_i$ to $U_i$ for $i=1,2$. Under the common-support assumption for the sampling and reinforcement graphs, and after a suitable ordering of the vertices, the matrices $\widetilde A$ and $W$ have the corresponding block forms
\[
WQ\tA
=
\begin{pmatrix}
0&W_1\\
W_2&0
\end{pmatrix}
\begin{pmatrix}
Q_1&0\\
0&Q_2
\end{pmatrix}
\begin{pmatrix}
0&\widetilde A_1\\
\widetilde A_2&0
\end{pmatrix}  = 
\begin{pmatrix}
W_1Q_2\widetilde A_2&0\\
0&W_2Q_1\widetilde A_1
\end{pmatrix}.
\]
Hence $\det(WQ\tA-I) = \det(W_1Q_2\widetilde A_2-I_{|U_1|})\det(W_2Q_1\widetilde A_1-I_{|U_2|})$. This is similar to what we get for the even cycle case, except that it is difficult to characterize (in terms of graph structure) when either $\det(W_2Q_1\widetilde A_1-I_{|U_2|})$ or $\det(W_1Q_2\widetilde A_2-I_{|U_1|})$ is zero.

We close this section with the following remark on weakening the condition of strongly connectedness of the underlying graph.
\begin{remark}[Weakly connected graphs]\label{rem:Weak}
The results in this section can be extended from strongly
connected graphs to weakly connected graphs. We briefly indicate the extension.
\begin{itemize}
    \item Suppose first that there is a non-empty set of {\it stubborn nodes}
    $\mathcal S\subseteq \mathcal V$, where a node $i\in\mathcal S$ is called
    stubborn if $d_i^{\mathrm{in}}=0$. If $S \neq \emptyset$ and the subgraph induced by
    $\mathcal V\setminus \mathcal S$ is strongly connected, then the conclusion of Theorem~\ref{thm:deterministic_limit} continues to hold with the limit on $V \setminus S$ having the form different from equation\eqref{deter_limit}. The configurations of the nodes in $S$ remain unchanged over time. 


    \item Suppose next that there are no stubborn nodes, but the graph is
    still weakly connected. Decompose $\mathcal G$ into its strongly connected
    components. Using an argument similar to Corollary~1 of \cite{kaur2025interacting}, it
    is enough to check the deterministic-limit condition on the stubborn
    strongly connected components. Here, a stubborn strongly connected
    component is a strongly connected subgraph induced by a set
    $F\subseteq\mathcal V$ such that there is no $j\notin F$, $v\in F$
    with $j\to v$.
\end{itemize}
\end{remark}

\subsection{Fluctuations in the Deterministic Regime}
\label{sec:LILCLT}

Throughout this subsection, we consider the general model
of Section~\ref{sec:model}, allowing distinct sampling and
reinforcement graphs, and assume that
$H=WQ\widetilde A-I$ is invertible.
Theorem~\ref{thm:deterministic_limit} then gives
$\Z_t\to\bZ^*$ almost surely, where $\bZ^*H+C=0$.
Here we state the  central limit theorem and a law of the iterated logarithm and 
the proofs are given in Section~\ref{sec:Proofs_1}. Defining $X_t=\Z_t-\bZ^*$, the recursion \eqref{eq:main_recur_vec} becomes \begin{equation} \label{eq:centered_SA_recursion} X_{t+1} = X_t + \frac{1}{t+1} \bigl( X_tH+\Delta M_{t+1}+\epsilon_{t+1} \bigr). \end{equation} 
Throughout this subsection, we allow the perturbed recursion
\eqref{eq:gen_recur_perturb} and assume that its remainder
satisfies
\begin{equation}\label{eq:error_one_step_bound}
\|\epsilon_t\|=O(t^{-1})
\qquad\text{almost surely}.
\end{equation}
For the unperturbed recursion \eqref{eq:gen_recur},
$\epsilon_t\equiv0$.
The remainder is retained to include normalized urn processes:
rewriting their exact step sizes using $1/(t+1)$ produces
a generally nonzero remainder satisfying this bound,
as shown in Proposition~\ref{prop:urn}.
Its contribution to a single update is therefore
$O(t^{-2})$ almost surely.

\medskip 
\noindent 
Let $B:=WQ\widetilde A$ and assume that $B$ is
diagonalizable over $\mathbb C$.
Denote its eigenvalues by $\lambda_1,\ldots,\lambda_N$,
ordered so that
\[
\operatorname{Re}(\lambda_1)
\ge\operatorname{Re}(\lambda_2)
\ge\cdots\ge\operatorname{Re}(\lambda_N).
\]
Write
\[
B=U\Lambda U^{-1},
\qquad
\Lambda=\operatorname{diag}(\lambda_1,\ldots,\lambda_N).
\]
Let $Y_t:=X_tU$ and, for each $r\in[N]$, set
$\lambda_r':=\lambda_r-1$.
Thus $\lambda_r'$ is the corresponding eigenvalue of
the drift matrix $H=B-I$.

In Section~\ref{sec:convergence_random}, the eigenvalues
of $I-Q\widetilde A$ are denoted by $\nu_r$.
When $W=I$, the conventions are related by
$\nu_r=1-\lambda_r$, after matching indices.

For positive sequences $(a_n)$ and $(b_n)$, we write
$a_n\sim b_n$ if $a_n/b_n\to1$, and $a_n\asymp b_n$
if there exist constants $c_1,c_2>0$ such that
$c_1b_n\le a_n\le c_2b_n$ for all sufficiently large $n$.

The next theorem establishes almost sure convergence rates toward the deterministic limit. The proof is based on coordinate-wise estimates for the projected process $Y_t=X_tU$, together with the spectral decomposition of $B$. Depending on the real parts of the eigenvalues of $B$, three regimes emerge: diffusive, critical, and super-diffusive. The law of the iterated logarithm and the central limit theorem are obtained from general stochastic approximation results of Mokkadem and Pelletier \cite{MP2005} and Zhang \cite{MR3582813}; the required auxiliary estimates are given in Appendix~\ref{app:appendix}.

\begin{theorem}[Projected martingale rates]
\label{thm:projected_martingale_rates}
For every \(\varepsilon>0\) and every \(r\in[N]\), we have
\[ 
Y_t(r) = \begin{cases} o\!\left(t^{-1/2}(\log t)^{1/2+\varepsilon}\right), & \operatorname{Re}(\lambda_r)<\frac12, \\[6pt]
o\!\left( t^{-1/2} (\log t)^{1/2} (\log\log t)^{1/2+\varepsilon} \right), & \operatorname{Re}(\lambda_r)=\frac12, \\[6pt] O\!\left(t^{\operatorname{Re}(\lambda_r)-1}\right), & \frac12<\operatorname{Re}(\lambda_r)<1 , 
\end{cases} \qquad\text{a.s.} 
\]
Consequently
\[
\|X_t\|
=
\begin{cases}
o\!\left(t^{-1/2}(\log t)^{1/2+\varepsilon}\right),
& \text{if } \operatorname{Re}(\lambda_1)<\frac12,\\[6pt]
o\!\left(t^{-1/2}(\log t)^{1/2}
(\log\log t)^{1/2+\varepsilon}\right),
& \text{if } \operatorname{Re}(\lambda_1)=\frac12,\\[6pt]
O\!\left(t^{\operatorname{Re}(\lambda_1)-1}\right),
& \text{if } \frac12<\operatorname{Re}(\lambda_1)<1.
\end{cases}
\qquad \text{a.s.}
\]
\end{theorem}
The three cases correspond respectively to the diffusive, critical, and superdiffusive spectral regimes.

\begin{remark}
The joint LIL and CLT below are stated for deterministic limits.  The proof of Theorem~\ref{thm:projected_martingale_rates} remains valid whenever $\Z_t\xrightarrow{\mathrm{a.s.}}\bZ^*$ for a (possibly random) limit $\bZ^*$ satisfying $\bZ^*H+C=0$  almost surely. 

\end{remark}

For each $j\in[N]$, define
\[
s_j^\star:=\sum_{k\in\VV}w_{kj}\bZ^*(k),
\qquad
p_j^\star:=q_js_j^\star+(1-q_j)(1-s_j^\star).
\]
Set
\[
\Gamma'
:=\operatorname{diag}\bigl(
p_1^\star(1-p_1^\star),\ldots,
p_N^\star(1-p_N^\star)
\bigr),
\qquad
\Gamma:=\widetilde A^\top\Gamma'\widetilde A.
\]
By Lemma~\ref{lem:limiting_martingale_covariance},
$\Gamma$ is the limiting conditional covariance of
$\Delta M_{t+1}$.
For $\operatorname{Re}(\lambda_1)<1/2$, define
\begin{equation}\label{eq:Sigma_infty_LIL}
        \Sigma_\infty
        :=
        \int_0^\infty
        \exp\!\left((B-\tfrac12 I)^\top u\right)
        \Gamma
        \exp\!\left((B-\tfrac12 I)u\right)
        \,du.
\end{equation}
The next theorem establishes a joint law of the iterated logarithm for the centered process.  

\medskip
\begin{theorem}[Joint compact LIL]
\label{thm:joint_compact_LIL_SA}
Assume that \(\operatorname{Re}(\lambda_1)<1/2\), that \(\Z_t\) admits a deterministic limit, and that \(\Gamma\) is positive definite. Then, with probability one, the sequence $ \left\{
        \sqrt{\frac{t}{2\log\log t}}\,
        (\Z_t-\bZ^*)
        \right\}_{t\ge 3}$
is relatively compact in $\IR^N$, and its set of limit points is the ellipsoid  $\mathcal E_{\mathrm{LIL}}
        =
        \left\{
        x\in \IR^N:
        x\Sigma_{\infty}^{-1}x^\top \le 1
        \right\}.$
\end{theorem}

\medskip
\begin{remark} 
The matrix \(\Gamma'\) is diagonal with entries $\Gamma'_{jj}=p_j^\star(1-p_j^\star)$, the matrix \(\Gamma'\) is positive definite whenever $ 0<p_j^\star<1$ for every $j\in[N]$. This condition holds, for example, if \(q_j\in(0,1)\) for every \(j\). More
generally, it is enough that $$0<
        q_js_j^\star+(1-q_j)(1-s_j^\star)
        <1
        \qquad\text{for every }j\in[N].$$

If, in addition, \(\widetilde A\) is nonsingular, then \(\Gamma\) is positive definite. Indeed, for every nonzero \(x\in\IR^N\),  
$x^\top \Gamma x  = (\widetilde A x)^\top \Gamma'(\widetilde A x)>0$, because \(\widetilde A x\neq0\) and \(\Gamma'\) is positive definite.
\end{remark}

\begin{theorem}[Joint Central Limit Theorem]
\label{thm:joint_CLT_SA}
Assume that $\Z_t$ admits an almost sure deterministic limit $\bZ^*$ and that
$\operatorname{Re}(\lambda_1)<\frac12$.
Then, as $t\to\infty$,
\[
        \sqrt t\,(\Z_t-\bZ^*)
        \xrightarrow{d}
        \mathcal N(0,\Sigma_\infty),        
\]
where \(\Sigma_\infty\) is defined in
\eqref{eq:Sigma_infty_LIL}. The limiting distribution may be
degenerate.
\end{theorem}





\section{ Random Limit Regime}\label{sec:convergence_random}
Section~\ref{sec:convergence_deter} gave structural conditions under which $\Z_t$ converges almost surely to a deterministic limit. When these conditions fail, the limiting vector, when it exists, is expected to be random. In this section, we study one such random-limit regime. Throughout the section we consider the following model.

\

\noindent{\bf (A2).}
We take $W=I$ and $q_i\in\{0,1\}$ for every $i\in\VV$.
The reinforcement graph $\GG=(\VV,\EE)$ is strongly connected,
and all vertices are positive-critical; that is,
\[
\mu_{ji}>0\quad\text{for every }j\to i,
\qquad
\sum_{j\to i}\mu_{ji}=1\quad\text{for every }i\in\VV.
\]
We also assume that $c_{ji}=0$ on every reinforcement edge
and that $\epsilon_t(i)=0$ for every $i\in\VV$ and $t\geq0$.

\

Throughout this section we restrict attention to graphs that admit a GPA partition and  prove almost sure convergence and describe the structure of the random limiting vector. A special case of this for interacting urn processes was studied in \cite{kaur2025interacting}. The random limit structure conjectured there becomes a corollary of our results.

Recall that, $\PP=\{i:q_i=1\}$ and $\DD=\{i:q_i=0\}$. Since $W=I$, the relevant partition is the original one-step graph partition algorithm of \cite{kaur2025interacting}. Fix such a partition and write $ \PP=P_1\cup P_2,\DD=D_1\cup D_2$, where the four sets $P_1,P_2,D_1,D_2$ are mutually disjoint. We call these four sets blocks. Let
$p_1=|P_1|$, $p_2=|P_2|$, $d_1=|D_1|$, and $d_2=|D_2|$. Without loss of generality, label the vertices so that the first $p_1$ vertices belong to $P_1$, the next $p_2$ vertices belong to $P_2$, the next $d_1$ vertices belong to $D_1$, and the last $d_2$ vertices belong to $D_2$. For a matrix $B$ and subsets $U_1,U_2\subseteq[N]$, we write $B_{U_1,U_2}$ for the submatrix of $B$ obtained by restricting the row indices to $U_1$ and the column indices to $U_2$. We write $B_U$ for $B_{U,U}$. For graphs that admit a partition, the  adjacency matrix has the following form:
\begin{equation}\label{adjacency-partition}
    A = \begin{pmatrix}
        A_{P_1,P_1} &0 & 0 & A_{P_1, D_2}\\
        0 & A_{P_2,P_2} & A_{P_2, D_1} & 0\\
        A_{D_1,P_1} &0 & 0 & A_{D_1, D_2}\\
        0 & A_{D_2,P_2} & A_{D_2, D_1} & 0
    \end{pmatrix}.
\end{equation}
With this labelling, we can write
\begin{equation}\label{eq:block_Q}
    Q=
\begin{pmatrix}
I_{\mathcal{P}} & \bzero\\
\bzero & -I_{\mathcal{D}}
\end{pmatrix},
\end{equation}
where \(I_{\mathcal{P}}\) and \(I_{\mathcal{D}}\) denote the identity matrices on the
coordinates indexed by \(\PP\) and \(\DD\), respectively, and \(\bzero\)
denotes a zero block of the appropriate size.  In this case the sampling indicators defined in \eqref{eq:sampling_indicator} reduces to
\begin{align}\label{chi}
\mathbb P(\I_{t+1}(i)=1\mid\FF_t)
=
\begin{cases}
Z_t(i), & i\in\PP,\\
1-Z_t(i), & i\in\DD.
\end{cases}
\end{align}
\noindent
and under $\mathrm{\textbf{(A2)}}$ the recursion \eqref{eq:gen_recur} can be written as 
\begin{eqnarray}
\label{noderecurnew}
Z_{t+1}(i)
&=& Z_t(i)\bigg(1-\frac{1}{t+1}\bigg)
+ \frac{1}{t+1} \bigg( \sum_{\substack{j\in N(i)\cap\mathcal P}} \mu_{ji}\I_{t+1}(j) + \sum_{\substack{j\in N(i)\cap\mathcal D}} \mu_{ji}\I_{t+1}(j) \bigg).
\end{eqnarray}

Two extreme cases are worth noting. If $\DD=\emptyset$, then every vertex is preferential and positive-critical, and the graph forms a single strongly connected component. At the opposite extreme, if $\PP=\emptyset$ and the graph admits a GPA partition, then the partition reduces to a bipartite structure, with all directed edges running between the two parts and no directed edges within either part.

Throughout this section, for any non-empty set $U\subseteq\VV=[N]$, define  $Z_t^U = \sum_{i\in U} Z_t(i)$ and $\bar{Z}_t^U = \frac{1}{|U|} Z_t^U$. We call $\bar{Z}_t^U$ the mean state variable in the set $U$. For any $i \in [N]$ and $U\subseteq \VV$, define
$d^{out}_{i,U}=\sum_{j\in U}\mathbf{1}_{\{i\to j\}}$ and
$d^{in}_{i,U}=\sum_{j\in U}\mathbf{1}_{\{j\to i\}}$. Thus, for every $i$,
$d^{out}_{i,\mathcal{D}}+d^{out}_{i,\mathcal{P}}=d_i^{out}$ and
$d^{in}_{i,\mathcal{D}}+d^{in}_{i,\mathcal{P}}=d_i^{in}$. Finally, for $U\subseteq[N]$, let $\bone_U$ denote the row indicator vector of $U$,
that is, $(\bone_U)_i=1$ if $i\in U$ and $(\bone_U)_i=0$ otherwise. We write
$\bone=\bone_{[N]}$ for the all-one row vector.

\subsection{Block-Level Dynamic}\label{sec:mean_fractions}

\subsubsection{Limiting Block Averages}
In this section, we explore how the limiting mean fractions of the four blocks are related. 
We require the following balance condition.

\begin{itemize}

 \item[{\bf (B)}] {\bf Uniform outgoing reinforcement:} there exist
$\mu^{out}_{\mathcal{P}}\geq 0$ and $\mu^{out}_{\mathcal{D}}\geq 0$ such that
\begin{equation}\label{eq:balance_condition2}
    \widetilde A\bone_{\mathcal{P}}^{\top}
=
\mu^{out}_{\mathcal{P}}\bone^{\top}
\quad \text{and} \quad
\widetilde A\bone_{\mathcal{D}}^{\top}
=
\mu^{out}_{\mathcal{D}}\bone^{\top}. 
\end{equation}

\end{itemize}
Condition {\bf (B)} guarantees that the total reinforcement sent toward preferential and de-preferential vertices depends only on the type of the sender and not on the specific vertex. 
\begin{theorem}\label{thm:AS_average}
Assume {\bf (B)} holds and
$\mu^{out}_{\mathcal P}-\mu^{out}_{\mathcal D}<1$. Then, as $t\to\infty$,
\begin{equation}
\big(
Z_t^{\mathcal P} , 
Z_t^{\mathcal D}
 \big)
\xrightarrow{a.s.}
\frac{d_1+d_2}
{\mu^{out}_{\mathcal D}-\mu^{out}_{\mathcal P}+1}
\big(
\mu^{out}_{\mathcal P}, 
\mu^{out}_{\mathcal D}
\big).
\end{equation}
Further, $\left( Z_t^{P_1}+Z_t^{D_1}, Z_t^{P_1}+Z_t^{P_2}, Z_t^{D_1}+Z_t^{D_2} \right)  and  \left( Z_t^{P_2}+Z_t^{D_2}, Z_t^{P_1}+Z_t^{P_2}, Z_t^{D_1}+Z_t^{D_2} \right)$ converge almost surely to $\mu^{out}(d_1+d_2)(1,1,1)$ and $Z_t^{P_1},Z_t^{P_2},Z_t^{D_1}$ and $Z_t^{D_2}$ admit almost sure limits, provided $\mu^{out}_{\mathcal P} =\mu^{out}_{\mathcal D} = \mu^{out}$.

\end{theorem}

\begin{corollary}\label{cor:Z_limit}
Assume {\bf (B)} holds, and suppose
$\mu^{out}_{\mathcal P}=\mu^{out}_{\mathcal D}$.
Write $\bar Z_t=N^{-1}\sum_{i=1}^N Z_t(i)$.
Then, as
$t\to\infty$, $\bar Z_t\xrightarrow{a.s.}1/2$, provided
$p_1+p_2=d_1+d_2$. Moreover, if $p_1=p_2=d_1=d_2$, then, as $t\to\infty$,
\[
\left(\bar Z_t^{P_1}+\bar Z_t^{D_1},
\bar Z_t^{P_1}+\bar Z_t^{P_2},
\bar Z_t^{D_1}+\bar Z_t^{D_2},
\bar Z_t^{P_2}+\bar Z_t^{D_2} \right)
\xrightarrow{a.s.}
(1,1,1,1).
\]
\end{corollary}
The above corollary shows that, in the equal-size case, if
$\bar Z_t^{P_1}$ converges almost surely to a random variable $Z_\infty$ as
$t\to\infty$, then $\bar Z_t^{P_2}$ and $\bar Z_t^{D_1}$ converge almost surely
to $1-Z_\infty$, while $\bar Z_t^{D_2}$ converges almost surely to
$Z_\infty$.

\subsubsection{Within-Block Synchronization}

In this section, we study synchronization within the four blocks of the GPA
partition. Write $\mathcal B=\{P_1,P_2,D_1,D_2\}$. Under additional conditions,
we show that, for every $U\in\mathcal B$, $Z_t(i)-\bar Z_t^U\to 0$ almost
surely for all $i\in U$. We impose the following conditions. 
\begin{itemize}
\item[{\bf(C)}] {\bf Constant-block action:} For every $U\in\mathcal B$, we assume
that $\bone_U \widetilde A \in \operatorname{span}\{\bone_V:V\in\mathcal B\}$. That is,
there exist constants $\mu_{UV}\geq 0$ such that
$\bone_U \widetilde A =\sum_{V\in\mathcal B} \mu_{UV}\bone_V$ for every $U\in\mathcal B$.
\item[{\bf (D)}] {\bf Diagonalizability:} The matrix $Q\widetilde A$ is diagonalizable. That is, there exists a matrix $P$ such that $I-  Q\widetilde A = P \Lambda P^{-1}$ where $\Lambda = Diag (\nu_1, \dots, \nu_N)$. Note that the rows of $P^{-1}$ are the corresponding left eigenvectors of $I-  Q \widetilde A$, denoted by $v_1, v_2, \dots, v_N$.
\end{itemize}
When $W=I$, assumption (D) is the same diagonalizability
requirement imposed in Section~\ref{sec:LILCLT}.
We state it explicitly here because that earlier
assumption applied only to the fluctuation results.
Define the block-averaging projection
\[
\tilde J
=
\sum_{\substack{U\in\mathcal B\\ |U|>0}}
\frac{1}{|U|}\bone_U^\top\bone_U,
\]
and put
\[
\hat v=\bone_{P_1\cup D_2}-\bone_{P_2\cup D_1}.
\]
Throughout, sums involving inverse block sizes are taken
over nonempty blocks.
An immediate consequence of constant-block action is
the following. \begin{lemma}\label{lem:commute}
Assume {\rm \bf(B)} and {\rm \bf(C)} hold. Then $\tilde J$ commutes with $Q\widetilde A$.
\end{lemma}

\begin{proof}
For $U\in\mathcal B$, let
$\sigma_U=\begin{cases}1,& U=P_1\text{ or }P_2,\\ -1,& U=D_1\text{ or }D_2.\end{cases}$
Since $\bone_UQ=\sigma_U\bone_U$,  {\rm \bf (C)} gives 
$\bone_UQ\widetilde{A}=\sigma_U\bone_U\widetilde{A}
=\sigma_U\sum_{V\in\mathcal B}\mu_{UV}\bone_V$. Recall that
$\tilde J=\sum_{U\in\mathcal B}|U|^{-1}\bone_U^\top\bone_U$. Hence,
\[
\tilde JQ\widetilde{A}
=
\sum_{U,V\in\mathcal B}
\frac{\sigma_U \mu_{UV}}{|U|}\bone_U^\top\bone_V.
\]
Next, we compute $Q \widetilde{A} \tilde J$. Since
$\tilde J=\sum_{V\in\mathcal B}|V|^{-1}\bone_V^\top\bone_V$, we have
\[
Q \widetilde{A} \tilde J
=
\sum_{V\in\mathcal B}
\frac{1}{|V|}
(Q \widetilde{A} \bone_V^\top)\bone_V.
\]
For $i\in U$, the $i$-th entry of $Q \widetilde{A} \bone_V^\top$ equals
$\sigma_U \sum_{\substack{j:\,i\to j\\ j\in V}}\mu_{ij}$. By {\rm \bf (B)} and the GPA block structure,
$\sum_{\substack{j:\,i\to j\\ j\in V}}\mu_{ij}$ is constant over $i\in U$; denote this constant by
$r_{UV}$. Therefore
\[
Q \widetilde{A} \tilde J
=
\sum_{U,V\in\mathcal B}
\frac{\sigma_Ur_{UV}}{|V|}\bone_U^\top\bone_V.
\]
It remains to compare the constants in the two expressions. The total mass in
the $U,V$ block of $\widetilde A$ is $\sum\limits_{\substack{i\in U,\ j\in V\\ i\to j}}\mu_{ij}$. By {\rm \bf (C)}, each column in $V$ receives weight $\mu_{UV}$ from rows in $U$, so
this total is $|V|\mu_{UV}$. By the definition of $r_{UV}$, each row in $U$
sends weight $r_{UV}$ to columns in $V$, so the same total is
$|U|r_{UV}$. Hence $|V|\mu_{UV}=|U|r_{UV}$, or equivalently
$\mu_{UV}/|U|=r_{UV}/|V|$. The two displayed expressions are therefore equal, and
so $\tilde JQ\widetilde{A}=Q \widetilde{A} \tilde J$.
\end{proof}
We now state the main convergence theorem of this section. 
\begin{theorem}\label{thm:randomlimit}
Suppose that {\bf (B)}, {\bf (C)}, and {\bf (D)}
hold and that $p_1+p_2=d_1+d_2$. Then $\Z_t(I-\tilde J)\to{\bf 0}$ almost surely as $t\to\infty$.
\end{theorem}

The proof of the theorem relies on the spectral structure
of the matrix $Q\widetilde A-I$, stated in Lemma \ref{lem:eigenvector}. We also need the following results, which are proved in section~\ref{sec:Proofs_2}.

\begin{proposition}\label{prop:Var}
Suppose {\rm \bf (D)} holds. Let $U$ be an $N\times N$ matrix such that
$\mathbf{1}_{\mathcal D}\widetilde{A} U=0$, $Q\widetilde{A}U=UQ\widetilde{A}$, and $\hat vU=0$. Let
$\nu_1,\ldots,\nu_N$ be the eigenvalues of
$I-Q\widetilde{A}$, and set
$\nu^*:=\min\{\operatorname{Re}(\nu_i):\nu_i\neq 0\}>0$. Then
\begin{itemize}
    \item $\IE[\Z_tU]=\mathcal{O}\left(t^{-\nu^*}\right)$.
    \item $\operatorname{Var}(\Z_tU)=\mathcal{O}\left(t^{-\delta}\right)$ for some $\delta>0$.
\end{itemize}
\end{proposition}

Before we show synchronization on each part of the partition, we prove that $\Z_t$ admits an almost sure limit. 
\begin{theorem}[Almost Sure Convergence of $\Z_t$]\label{thm:QM1} Suppose {\rm \bf (B)}, {\rm \bf (C)} and {\rm \bf (D)} hold and $p_1+p_2=d_1+d_2$. Then $\Z_t$ has an almost sure limit. \end{theorem} 

We are now ready to prove Theorem~\ref{thm:randomlimit}.

\begin{proof}[Proof of Theorem~\ref{thm:randomlimit}]
Let $\Phi_t=\Z_t(I-\tilde J)$. We first apply Proposition~\ref{prop:Var} with $U=I-\tilde J$. Its assumptions are satisfied because $\mathbf 1_{\mathcal D}\widetilde{A}$ is constant on the four parts of the partition by (C), so $\mathbf 1_{\mathcal D}\widetilde{A}(I-\tilde J)=0$; $\tilde J$ commutes with $Q\widetilde{A}$ by Lemma~\ref{lem:commute}, so $Q\widetilde{A}(I-\tilde J)=(I-\tilde J)Q\widetilde{A}$; and $\hat v$ is constant on the four parts of the partition, so $\hat v(I-\tilde J)=0$. Hence Proposition~\ref{prop:Var} gives $\Phi_t\to{\bf 0}$ in $\mathcal L^2$.

On the other hand, by Theorem~\ref{thm:QM1}, $\Z_t$ has an almost sure limit; denote it by $\Z^*$. Therefore $\Phi_t=\Z_t(I-\tilde J)$ converges almost surely to $\Z^*(I-\tilde J)$. Since $\Z_t$ is bounded, $\|\Phi_t\|^2$ is bounded, and bounded convergence together with $\IE[\|\Phi_t\|^2]\to0$ gives $\IE[\|\Z^*(I-\tilde J)\|^2]=0$. Thus $\Z^*(I-\tilde J)={\bf 0}$ almost surely, and consequently $\Phi_t\to{\bf 0}$ almost surely. This proves that $\Z_t(I-\tilde{J})\to0$ almost surely.
\end{proof}

We note that the above result can alternatively be established using the theory of almost supermartingales. We take this approach to prove the same result for cycle graphs. Due to their special structure, several of the assumptions imposed in this section are automatically satisfied. For this reason, we illustrate the case of cycle graphs separately in Appendix~\ref{sec:extra}.

\section{Proofs}\label{sec:proofs}

\subsection{Proofs of Fluctuation Results from Section~\ref{sec:LILCLT}} \label{sec:Proofs_1}
\begin{proof}[Proof of Theorem \ref{thm:projected_martingale_rates}]
Multiplying \eqref{eq:centered_SA_recursion} on the right by $U$ gives
\[
        Y_{t+1}
        =
        Y_t\left(I+\frac{\Lambda-I}{t+1}\right)
        +
        \frac{1}{t+1}\left(\Delta M_{t+1}U+\epsilon_{t+1}U\right).
\]
Hence, for each \(r\in[N]\),
\begin{equation}\label{eq:projected_recursion}
        Y_{t+1}(r)
        =
        \left(1+\frac{\lambda_r'}{t+1}\right)Y_t(r)
        +
        \frac{(\Delta M_{t+1}U)(r)+(\epsilon_{t+1}U)(r)}{t+1}.
\end{equation}

Choose a fixed integer $t_0$ sufficiently large so that
$1+\lambda_r'/k\ne 0$ for all $k\ge t_0+1$ and all $r$. This avoids possible zero factors in the Euler products. Put \(d_{t_0}(r)=1\)
and, for \(t\geq t_0+1\),
\(d_t(r):=\prod_{j=t_0+1}^{t}(1+\lambda_r'/j)\). By the gamma-function asymptotic,
\begin{equation}\label{eq:dt_asymp_projected}
        d_t(r)
        =
        \frac{\Gamma(t+1+\lambda_r')\Gamma(t_0+1)}
             {\Gamma(t_0+1+\lambda_r')\Gamma(t+1)}
        \asymp
        t^{\mathrm{Re}(\lambda_r')}.\end{equation}
Iterating \eqref{eq:projected_recursion} gives the following.
\begin{equation}\label{eq:projected_variation_constants}
        Y_t(r)
        =
        d_t(r)
        \left(
            Y_{t_0}(r)
            +
            \mathcal L_t(r)
        \right)
        +
        \mathcal E_t(r),
\end{equation}
where $\mathcal L_t(r)
        :=
        \sum_{k=t_0}^{t-1}
        \frac{(\Delta M_{k+1}U)(r)}
             {(k+1)d_{k+1}(r)}$
and
        $\mathcal E_t(r)
        :=
        d_t(r)
        \sum_{k=t_0}^{t-1}
        \frac{(\epsilon_{k+1}U)(r)}
             {(k+1)d_{k+1}(r)}$.  We analyse separately the martingale term $\mathcal L_t(r)$ and the accumulated error term $\mathcal E_t(r)$.

\begin{itemize}
\item[Step 1:] Since the martingale difference $\Delta M_t$ is  bounded and $U$ is fixed,
there is a deterministic constant $K>0$ such that $| (\Delta M_t U)(r)|\le K$ for all $t,r$. Hence the predictable quadratic variation of $\mathcal L_t(r)$ satisfies
\[
        \big\langle \mathcal L(r)\big\rangle_t
        :=
        \sum_{k=t_0}^{t-1}
        \IE\left[
        \left.
        \left|
        \frac{(\Delta M_{k+1} U)(r)}
             {(k+1)d_{k+1}(r)}
        \right|^2
        \right|\FF_k
        \right]
        \le
        K^2\,\ell_t(r),
\]
where
$\ell_t(r)
        :=
        \sum_{k=t_0}^{t-1}
        \left|(k+1)d_{k+1}(r)\right|^{-2}.$ Using \eqref{eq:dt_asymp_projected}, we immediately get that $\ell_t(r)  \asymp  \sum_{k=t_0}^{t}  k^{-2-2\mathrm{Re}(\lambda_r')}$. Therefore
\[
\ell_t(r)
=
\begin{cases}
O\!\left(t^{-1-2\mathrm{Re}(\lambda_r')}\right),
& \mathrm{Re}(\lambda_r')<-1/2,\\[6pt]
O(\log t),
& \mathrm{Re}(\lambda_r')=-1/2,
\end{cases}
\]
and $\ell_t(r) \to \ell_\infty(r) \coloneqq \sum_{k=t_0}^{\infty}  \left|(k+1)d_{k+1}(r)\right|^{-2} < \infty$  for $\mathrm{Re}(\lambda_r')>-1/2$.

We now apply Theorem~\ref{aux:duflo_martingale_estimate} separately to the
real and imaginary parts of the complex martingale \(\mathcal L_t(r)\).
Since \(\langle \mathcal L(r)\rangle_t\le K^2\ell_t(r)\), it follows
that, whenever \(\ell_t(r)\to\infty\), for every \(\varepsilon>0\),
\begin{equation}\label{eq:L_bound_by_ell}
        \mathcal L_t(r)
        =
        o\!\left(
        (\ell_t(r))^{1/2}
        (\log \ell_t(r))^{1/2+\varepsilon}
        \right)
        \qquad\text{a.s.}
\end{equation}
If \(\ell_t(r)\) converges, then the same theorem gives $\mathcal L_t(r)
        \xrightarrow{a.s.}
        \mathcal L_\infty(r)$, where \(\mathcal L_\infty(r)\) is finite almost surely. In particular, this
convergence holds when \(\operatorname{Re}(\lambda_r')>-1/2\).

Multiplying \eqref{eq:L_bound_by_ell} by $d_t(r)$ gives, in the diffusive and critical cases,
\begin{equation}\label{eq:dL_rates}
        d_t(r)\mathcal L_t(r)
        =
        \begin{cases}
        o\!\left(
        t^{-1/2}(\log t)^{1/2+\varepsilon}
        \right),
        & \mathrm{Re}(\lambda_r')<-1/2,\\[8pt]
        o\!\left(
        t^{-1/2}
        (\log t)^{1/2}
        (\log\log t)^{1/2+\varepsilon}
        \right),
        & \mathrm{Re}(\lambda_r')=-1/2,
        \end{cases}
        \qquad\text{a.s.}
\end{equation}

\item[Step 2:]

We now estimate the accumulated error
$\mathcal E_t(r)$ using \eqref{eq:error_one_step_bound}.
Since $\|\epsilon_{k+1}\|=O((k+1)^{-1})$ almost surely and
$|d_t(r)|\asymp t^{\operatorname{Re}(\lambda_r')}$,
there exists an almost surely finite random constant $K>0$,
independent of $t$, such that
\[
        |\mathcal E_t(r)|
        \le
        K |d_t(r)|
        \sum_{k=t_0}^{t-1}
        \frac{1}{(k+1)^2 |d_{k+1}(r)|}
        \le
        K t^{\operatorname{Re}(\lambda_r')}
        \sum_{k=t_0}^{t}
        k^{-2-\operatorname{Re}(\lambda_r')}.
\]

It follows that
\begin{equation} \label{eq:errordecay}
        |\mathcal E_t(r)|
        =
        \begin{cases}
        O(t^{-1}),
        & \mathrm{Re}(\lambda_r')<-1,\\[6pt]
        O((\log t)/t),
        & \mathrm{Re}(\lambda_r')=-1,\\[6pt]
        O(t^{\mathrm{Re}(\lambda_r')}),
        & -1<\mathrm{Re}(\lambda_r').
        \end{cases}
\end{equation}
\end{itemize}
In particular, for $\mathrm{Re}(\lambda_r') \le -1/2$ this error is negligible when compared with the corresponding martingale rate in \eqref{eq:dL_rates}. Using the estimates from \eqref{eq:dL_rates}  in \eqref{eq:projected_variation_constants} (combined with the fact that the term $d_t(r)Y_{t_0}(r)$ is negligible compared with the
corresponding bound in \eqref{eq:dL_rates})  proves the diffusive and critical assertions.

\noindent In the superdiffusive case, \(-1/2<\operatorname{Re}(\lambda_r')<0\). So, 
\(2+\operatorname{Re}(\lambda_r')>1\). We therefore have $
        |\mathcal E_t(r)|
        =
        O\!\left(t^{\operatorname{Re}(\lambda_r')}\right)
        \text{a.s.}$
Further, 
\[
        \frac{Y_t(r)}{d_t(r)}
        =
        Y_{t_0}(r)
        +\mathcal L_t(r)
        +
        \sum_{k=t_0}^{t-1}
        \frac{(\epsilon_{k+1}U)(r)}
             {(k+1)d_{k+1}(r)}.
\]
 Recall that in this regime, $\mathcal L_t(r)
        \xrightarrow{a.s.}
        \mathcal L_\infty(r)$, where \(\mathcal L_\infty(r)\) is finite almost surely. Using
\(\|\epsilon_{k+1}\|=O((k+1)^{-1})\) and
\(|d_k(r)|\asymp k^{\operatorname{Re}(\lambda_r')}\), we get that the last series
converges absolutely. Therefore
\(Y_t(r)/d_t(r)\) converges almost surely to a finite random variable. Again using \(|d_t(r)|\asymp t^{\operatorname{Re}(\lambda_r')}\), we obtain
$Y_t(r)
        =
        O\!\left(t^{\operatorname{Re}(\lambda_r')}\right)
        \text{a.s.}$ This proves the assertion in the superdiffusive regime.

\noindent Finally, since \(U\) is invertible and \(X_t=Y_tU^{-1}\), there exists a
constant \(K_U>0\) such that $\|X_t\|\le K_U\|Y_t\|$.
Thus the coordinatewise projected bounds imply the asserted bounds for $\|X_t\|=\|\Z_t-\bZ^*\|$. This completes the proof.
\end{proof}

We next compute the limiting conditional covariance of the martingale increments.
\begin{lemma}[Limiting conditional covariance]\label{lem:limiting_martingale_covariance}
Assume that $\Z_t\to\bZ^*$ almost surely, where $\bZ^*$
may be random. Let $\Delta M_{t+1}$ be the martingale
difference appearing in \eqref{eq:main_recur_vec}. Then
\[
        \IE\!\left[
        (\Delta M_{t+1})^\top \Delta M_{t+1}\mid\FF_t
        \right]
        \xrightarrow{a.s.}
        \widetilde A^\top \Gamma'\widetilde A,
\]
where \(\Gamma'\) is a diagonal matrix with entries
$\Gamma'_{jj}=p_j^\star(1-p_j^\star)$, $p_j^\star=q_js_j^\star+(1-q_j)(1-s_j^\star)$ and 
        $s_j^\star
        =\sum_{k\in\VV}w_{kj}\bZ^*(k)$.
\end{lemma}
\begin{proof}
For each $j$, set $ s_j(t):=\sum_{k\in\VV}w_{kj}Z_t(k)$ and $p_j(t) \coloneqq q_j s_j(t)+(1-q_j)(1-s_j(t)) = (2q_j-1)s_j(t)+(1-q_j)$. Then, conditionally on $\FF_t$, the sampling indicator $\I_{t+1}(j)$ is Bernoulli with parameter $p_j(t)$. 

Define the row vector $\Delta M'_{t+1}
        :=
        \I_{t+1}
        -
        \IE[\I_{t+1}\mid\FF_t]$.
Then $\Delta M_{t+1}=\Delta M'_{t+1}\widetilde A$.  By conditional independence of the draws at a fixed time, \(\IE\!\left[
        (\Delta M'_{t+1})^\top \Delta M'_{t+1}
        \mid\FF_t
        \right]\) is
diagonal. Its \(j\)-th diagonal entry is 
        $p_j(t)(1-p_j(t))$. 
        
Since \(\Z_t\to\bZ^*\) almost surely, we have  $s_j(t)\xrightarrow{a.s.}s_j^\star
        :=
        \sum_{k\in\VV}w_{kj}\bZ^*(k)$. Hence $p_j(t)\xrightarrow{a.s.}p_j^\star
        :=
        q_js_j^\star+(1-q_j)(1-s_j^\star)$, and therefore
 $\IE\!\left[
        (\Delta M'_{t+1})^\top \Delta M'_{t+1}
        \mid\FF_t
        \right]
        \xrightarrow{a.s.}
        \Gamma' $, where \(\Gamma'\) is the diagonal matrix defined in the statement of the
lemma. Finally, since \(\Delta M_{t+1}=\Delta M'_{t+1}\widetilde A\),
\[
        \IE\!\left[
        (\Delta M_{t+1})^\top \Delta M_{t+1}
        \mid\FF_t
        \right]
        \xrightarrow{a.s.}
        \widetilde A^\top \Gamma'\widetilde A.
\]
This proves the lemma.
\end{proof}

\begin{proof}[Proof of Theorem\ref{thm:joint_compact_LIL_SA}]

We verify the assumptions of the stochastic-approximation compact LIL
stated in Theorem~\ref{aux:MP_compact_LIL} for the centred recursion
\eqref{eq:centered_SA_recursion}. This verification follows
\cite[Theorem~11]{das2024elephant}, with the remainder controlled
by \eqref{eq:error_one_step_bound}.
We use Theorem~\ref{aux:MP_compact_LIL} for the stochastic
approximation algorithm
\begin{equation}\label{eq:SA_LIL}
        X_{t+1}
        =
        X_t+\gamma_t\big[h(X_t)+r_{t+1}\big]
        +\sigma_t\xi_{t+1},
\end{equation}
with $h(x)=xH=x(B-I), r_{t+1}=\epsilon_{t+1}, \xi_{t+1}=\Delta M_{t+1}$
and $\gamma_t=\sigma_t=\frac{1}{t+1}$. Thus, 
\[
        v_t:=\frac{\gamma_t}{\sigma_t^2}=t+1,
        \qquad
        s_t:=\sum_{k=1}^{t}\gamma_k
        =
        \sum_{k=1}^{t}\frac1{k+1}
        \sim \log t.
\]
We now check assumptions {\bf (1)--(5)} of Theorem~\ref{aux:MP_compact_LIL}.

\smallskip

\noindent
\textbf{(1)}
Under the deterministic-limit assumption of the theorem, $\Z_t\xrightarrow{a.s.}\bZ^*$. Therefore, $X_t\xrightarrow{a.s.}0$. Thus the stochastic approximation algorithm in \eqref{eq:SA_LIL} is consistent, with equilibrium $0$.

\smallskip

\noindent
\textbf{(2)}
The function $h(x)=xH$ is linear, so its Jacobian at $0$
is $H=B-I$. Set $L:=1-\operatorname{Re}(\lambda_1)$.
Then
\[
\max_{\mu\in\mathrm{Sp}(H)}\operatorname{Re}(\mu)
=\operatorname{Re}(\lambda_1)-1=-L.
\]
Since $\operatorname{Re}(\lambda_1)<1/2$, we have
$L>1/2$. Therefore the linear drift is attractive.

\smallskip

\noindent
\textbf{(3)}
We clearly have  $\gamma_t\to0, \sigma_t\to0,
        s_t\to\infty$ and
        $\frac{\gamma_t}{\gamma_{t+1}}=1+O(\gamma_t)$. Moreover,
\[
        \frac1{\gamma_t}
        \left(
        1-\frac{v_{t-1}}{v_t}
        \right)
        =
        (t+1)\left(1-\frac{t}{t+1}\right)
        =
        1.
\]
Thus the parameter $\rho_{\mathrm{MP}}$ in
Theorem~\ref{aux:MP_compact_LIL} equals $1$. Since $L>1/2$, we have $1\in(0,2L)$, so {\bf (3)} holds.

\smallskip

\noindent
\smallskip

\noindent
\textbf{(4)}
The martingale
difference term in \eqref{eq:SA_LIL} has uniformly bounded increments. Hence, for every $\delta>0$, $\sup_{t\ge0} \IE\!\left[\|\Delta M_{t+1}\|^{2+\delta} \mid \FF_t \right] <\infty$ a.s. Therefore, by the sufficient martingale-difference condition for {\bf (4)} stated
in Comment~(4) preceding Theorem~1 of Mokkadem and Pelletier
\cite[Comment~(4)]{MP2005}, the required almost sure invariance principle
holds with covariance matrix $\Gamma$. Thus {\bf (4)} holds.

\smallskip

\noindent
\textbf{(5)} Using \eqref{eq:error_one_step_bound}, we have  $\|\epsilon_{t+1}\|= O(v_t^{-1/2})$ a.s. Thus {\bf (5)} holds by taking $r_{t+1}^{(1)}=\epsilon_{t+1}$ and the remaining
$o(\|X_t\|)$ term equal to zero.

\smallskip

\noindent All assumptions of Theorem~\ref{aux:MP_compact_LIL} are now verified. Therefore
\(\{\sqrt{v_t/(2\log s_t)}\,X_t\}_{t\ge1}\) is relatively compact almost surely,
and its limit set is the ellipsoid associated with
\[
        \int_0^\infty
        \exp\!\left((H+\tfrac12 I)^\top u\right)
        \Gamma
        \exp\!\left((H+\tfrac12 I)u\right)
        \,du .
\]
Since \(H=B-I\), the above matrix is precisely \(\Sigma_\infty\) defined in
\eqref{eq:Sigma_infty_LIL}. Also, \(v_t=t+1\) and \(s_t\sim\log t\), so the
normalization \(\sqrt{v_t/(2\log s_t)}\) is asymptotic to
\(\sqrt{t/(2\log\log t)}\). This proves the theorem.
\end{proof}
Finally we prove the joint Central Limit theorem. 
\begin{proof}[Proof of Theorem~\ref{thm:joint_CLT_SA}]
We apply Zhang's stochastic approximation CLT
\cite[Theorem~2.3]{MR3582813} to the centred recursion \eqref{eq:centered_SA_recursion}. In Zhang's notation, the drift is
\(h(x)=x(I-B)\), and hence the eigenvalues of \(Dh(0)\) are
\(1-\lambda_r\), \(1\le r\le N\). Therefore
\[
        \min_{1\le r\le N}\operatorname{Re}(1-\lambda_r)
        =
        1-\operatorname{Re}(\lambda_1)        >
        \frac12 .
\]
The deterministic-limit assumption gives \(X_t\to0\) almost surely. The
martingale difference term has uniformly bounded increments, so the required
Lindeberg condition holds. Moreover, by
Lemma~\ref{lem:limiting_martingale_covariance}, its conditional covariance
converges almost surely to \(\Gamma\). Finally, by
\eqref{eq:error_one_step_bound},
\[
        \sum_{k=1}^{t}\epsilon_k
        =
        O(\log t)
        =
        o(\sqrt t)
        \qquad\text{a.s.}
\]

Thus all assumptions of \cite[Theorem~2.3]{MR3582813} are satisfied, and the
conclusion follows. The covariance matrix appearing in that theorem
is exactly $\Sigma_\infty$, since
$-(Dh(0)-\tfrac12 I)=B-\tfrac12 I$. This proves the theorem.
\end{proof}

\subsection{Proofs of Random Regime Results from Section~\ref{sec:convergence_random}}
Throughout this subsection, we work with the random-limit
model and the fixed GPA partition introduced in
Section~\ref{sec:convergence_random}.
Additional assumptions are stated separately for each result.
\subsubsection{Proof of Theorem \ref{thm:AS_average}} \label{sec:Proofs_2}

We begin with the following lemma.
\begin{lemma}\label{lem:GPA_structure}
Assume {\bf (B)} holds. Then
$\mu^{out}_{\mathcal P}+\mu^{out}_{\mathcal D}=1$
and $\widetilde A\bone^\top=\bone^\top$.
Further, if $d_1=d_2$, then
$Z_t^{P_1}+Z_t^{D_2}-Z_t^{P_2}-Z_t^{D_1}$
is a bounded martingale.
\end{lemma}
 \begin{proof}
 Since every vertex is positive-critical, $\bone \widetilde A=\bone$. 
Multiplying the identity \[ \widetilde A\bone^{\top} = (\mu_{\mathcal P}^{out} +\mu_{\mathcal D}^{out}) \bone^{\top} \] from the left by $\bone$ gives 
 \[ N = \bone\widetilde A\bone^{\top} = (\mu_{\mathcal P}^{out} +\mu_{\mathcal D}^{out}) \bone\bone^{\top} = N(\mu_{\mathcal P}^{out} +\mu_{\mathcal D}^{out}). \] 
 Hence \[ \mu_{\mathcal P}^{out} +\mu_{\mathcal D}^{out} =1. \]
 Since $\bone  = \bone_{\PP} + \bone_{\DD}$ we get 
 \[\tA \bone^\top  = \mu^{out}_{\PP}\bone^{\top} + \mu^{out}_{\DD}\bone^{\top} = \bone^\top.\]
Now let $S=\bone_{P_1}^{\top}+\bone_{D_2}^{\top}
 -\bone_{P_2}^{\top}-\bone_{D_1}^{\top}$. Then
 $\Z_tS=Z_t^{P_1}+Z_t^{D_2}-Z_t^{P_2}-Z_t^{D_1}$.
 From \ref{eq:main_recur_vec} and {\rm \bf (A2)} we get, 
 \begin{eqnarray}\label{gen_exp}
 \IE[\Z_{t+1}S\mid\FF_t]
 &=&
 \Z_tS\left(1-\frac{1}{t+1}\right)
 +
 \frac{1}{t+1}
 \left( \Z_tQ\widetilde A+\bone_{\mathcal D}\widetilde A\right)S\\
 &=&
 \Z_tS+
 \frac{1}{t+1}
 \left[
 \Z_t(Q\widetilde A-I)S+\bone_{\mathcal D}\widetilde AS
 \right].
 \end{eqnarray}
 Therefore, it is enough to show that $(Q\widetilde A-I)S=\mathbf{0}$ and $\bone_{\mathcal D}\widetilde AS=\mathbf{0}$. Using the structure of the adjacency matrix (from \eqref{adjacency-partition}) we get 
 \[ \widetilde AS = \bone_{P_1}^{\top} -\bone_{D_2}^{\top} -\bone_{P_2}^{\top} +\bone_{D_1}^{\top}  \]
 and
 \[ Q\widetilde AS = \bone_{P_1}^{\top}  +\bone_{D_2}^{\top}-\bone_{P_2}^{\top} -\bone_{D_1}^{\top} = S. \]
 Similarly, we get $\bone_{\DD} \tA S = \mathbf{0}$ when $d_1 = d_2$. 
 \end{proof}
\begin{remark}\label{d1d2}
We note that $Z_t^{P_1}+Z_t^{D_2}-Z_t^{P_2}-Z_t^{D_1}$ is a bounded
supermartingale if $d_1\leq d_2$, and a bounded submartingale if $d_1\geq d_2$. Consequently, it admits an almost sure
limit for all $d_1,d_2$.
\end{remark}

\begin{proof}[Proof of Theorem \ref{thm:AS_average}]
From \eqref{noderecurnew} we get
\begin{align*}
Z_{t+1}(i)
&= Z_t(i) + \frac{1}{t+1} \left(\Delta M_{t+1}(i) +
\sum_{j\in N(i)\cap\mathcal P}\mu_{ji}Z_t(j) \right)
\\
&\quad
+ \frac{1}{t+1} \left( \sum_{j\in N(i)\cap\mathcal D}\mu_{ji}(1-Z_t(j))-Z_t(i)\right).
\end{align*}

For $U\subseteq\VV$, write
$\Delta M_{t+1}^{U}:=\sum_{i\in U}\Delta M_{t+1}(i)$.
By the block structure of the adjacency matrix and
{\rm \bf (B)}, we obtain
\begin{align*}
Z_{t+1}^{P_1}
&= Z_t^{P_1} + \frac{1}{t+1} \left[ \mu^{out}_{\mathcal P}
\left( Z_t^{P_1}-Z_t^{D_1}+d_1 \right) - Z_t^{P_1} \right] + \frac{\Delta M_{t+1}^{P_1}}{t+1},
\\
Z_{t+1}^{D_1}
&= Z_t^{D_1} + \frac{1}{t+1} \left[ \mu^{out}_{\mathcal D}
\left( Z_t^{P_2}-Z_t^{D_2}+d_2 \right) - Z_t^{D_1} \right] + \frac{\Delta M_{t+1}^{D_1}}{t+1},
\\
Z_{t+1}^{P_2}
&= Z_t^{P_2} + \frac{1}{t+1} \left[ \mu^{out}_{\mathcal P}
\left(Z_t^{P_2}-Z_t^{D_2}+d_2 \right) - Z_t^{P_2} \right]
+ \frac{\Delta M_{t+1}^{P_2}}{t+1},
\\
Z_{t+1}^{D_2}
&= Z_t^{D_2} + \frac{1}{t+1} \left[ \mu^{out}_{\mathcal D}
\left(Z_t^{P_1}-Z_t^{D_1}+d_1 \right) - Z_t^{D_2} \right]
+ \frac{\Delta M_{t+1}^{D_2}}{t+1}.
\end{align*}
Adding the $P_1$ and $P_2$ equations gives
\begin{align*}
Z_{t+1}^{\mathcal P}
&= Z_t^{\mathcal P} + \frac{1}{t+1} \left[ (\mu^{out}_{\mathcal P}-1)Z_t^{\mathcal P} -\mu^{out}_{\mathcal P}Z_t^{\mathcal D} +\mu^{out}_{\mathcal P}(d_1+d_2)
\right] + \frac{\Delta M_{t+1}^{\mathcal P}}{t+1},
\end{align*}
Similarly,
\begin{align*}
Z_{t+1}^{\mathcal D}
&= Z_t^{\mathcal D} + \frac{1}{t+1} \left[ \mu^{out}_{\mathcal D}Z_t^{\mathcal P} -(\mu^{out}_{\mathcal D}+1)Z_t^{\mathcal D} +\mu^{out}_{\mathcal D}(d_1+d_2)
\right] + \frac{\Delta M_{t+1}^{\mathcal D}}{t+1}.
\end{align*}
Thus, $(Z_t^{\mathcal P},Z_t^{\mathcal D})$ is a two-dimensional
stochastic approximation with limiting ODE
\begin{equation}\label{ode1}
\frac{d}{dt}
\begin{pmatrix}
Z_t^{\mathcal P}\\
Z_t^{\mathcal D}
\end{pmatrix}
=
\begin{pmatrix}
\mu^{out}_{\mathcal P}-1
&
-\mu^{out}_{\mathcal P}
\\
\mu^{out}_{\mathcal D}
&
-\mu^{out}_{\mathcal D}-1
\end{pmatrix}
\begin{pmatrix}
Z_t^{\mathcal P}\\
Z_t^{\mathcal D}
\end{pmatrix}
+
(d_1+d_2)
\begin{pmatrix}
\mu^{out}_{\mathcal P}\\
\mu^{out}_{\mathcal D}
\end{pmatrix}.
\end{equation}
The eigenvalues of the matrix in \eqref{ode1} are
$-1$ and
$\mu^{out}_{\mathcal P}
-\mu^{out}_{\mathcal D}-1$.
Hence, if
$\mu^{out}_{\mathcal P}
-\mu^{out}_{\mathcal D}<1$,
the equilibrium is stable and is given by
\[
\frac{d_1+d_2}
{\mu^{out}_{\mathcal D}-\mu^{out}_{\mathcal P}+1}
\begin{pmatrix}
\mu^{out}_{\mathcal P}\\
\mu^{out}_{\mathcal D}
\end{pmatrix}.
\]
The stochastic approximation convergence theorem therefore gives the claimed almost sure convergence. Now suppose
$\mu^{out}_{\mathcal P}
=
\mu^{out}_{\mathcal D}
=
\mu^{out}$.
Since
$\mu^{out}_{\mathcal P}+\mu^{out}_{\mathcal D}=1$,
we necessarily have
$\mu^{out}=\frac{1}{2}$.
Using the four recursions above, the process $\left( Z_t^{P_1}+Z_t^{D_1}, Z_t^{P_1}+Z_t^{P_2},
Z_t^{D_1}+Z_t^{D_2}
\right)$ satisfies a stochastic approximation scheme with drift function $\frac{d_1+d_2}{2}(1,1,1) + \left(
Z_t^{P_1}+Z_t^{D_1},
Z_t^{P_1}+Z_t^{P_2},
Z_t^{D_1}+Z_t^{D_2}
\right)\Gamma$,
where $\Gamma=
\begin{pmatrix}
-1 & 0 & 0\\
\frac12 & -\frac12 & \frac12\\
-\frac12 & -\frac12 & -\frac32
\end{pmatrix}$. The only eigenvalue of $\Gamma$ is $-1$. Hence
\[
\left(
Z_t^{P_1}+Z_t^{D_1},
Z_t^{P_1}+Z_t^{P_2},
Z_t^{D_1}+Z_t^{D_2}
\right)
\xrightarrow{a.s.}
\frac{d_1+d_2}{2}(1,1,1).
\]
Similarly,
\[
\left(
Z_t^{P_2}+Z_t^{D_2},
Z_t^{P_1}+Z_t^{P_2},
Z_t^{D_1}+Z_t^{D_2}
\right)
\xrightarrow{a.s.}
\frac{d_1+d_2}{2}(1,1,1).
\]
By Lemma~\ref{lem:GPA_structure} and
Remark~\ref{d1d2},
$Z_t^{P_1}+Z_t^{D_2}-Z_t^{P_2}-Z_t^{D_1}$ admits an almost sure limit.
Since $Z_t^{P_1}+Z_t^{P_2}$ has an almost sure deterministic limit, it follows that
$2Z_t^{P_1}+Z_t^{D_2}-Z_t^{D_1}$ admits an almost sure limit. Combining this
with the deterministic limit of $Z_t^{D_1}+Z_t^{D_2}$, we get that
$Z_t^{P_1}-Z_t^{D_1}$ admits an almost sure limit. Since
$Z_t^{P_1}+Z_t^{D_1}$ also has an almost sure deterministic limit, both
$Z_t^{P_1}$ and $Z_t^{D_1}$ admit almost sure limits. Similarly,
$Z_t^{P_2}$ and $Z_t^{D_2}$ also have almost sure limits.
\end{proof}

\medskip
\begin{remark}
    Observe that with assumption {\bf (B)} and Lemma~\ref{lem:GPA_structure}, the condition in Theorem~\ref{thm:AS_average} is equivalent $\mu^{out}_{\mathcal P}<1/2$ (equivalently, $\mu^{out}_{\mathcal D}<0$). 
\end{remark}
Recall the block-averaging projection $\tilde J$.
Under the assumptions of the final part of
Theorem~\ref{thm:AS_average},
the four block sums $Z_t^{P_1},Z_t^{P_2},Z_t^{D_1}$ and $Z_t^{D_2}$
admit almost sure limits. Hence the block-averaged vector
$\Z_t\tilde J$ converges almost surely to a random vector
$\Z_\infty\in[0,1]^N$.  When $\mu^{out}_{\mathcal P} = \mu^{out}_{\mathcal D}$, Theorem~\ref{thm:AS_average} gives
$Z_t^{\mathcal P}\xrightarrow[t\to\infty]{a.s.} \frac{d_1+d_2}{2}$ and $Z_t^{\mathcal D}\xrightarrow[t\to\infty]{a.s.} \frac{d_1+d_2}{2}$. Hence $Z_t^{\mathcal P}+Z_t^{\mathcal D}
\xrightarrow[t\to\infty]{a.s.} d_1+d_2$. Therefore, if
$\bar Z_t=\frac{1}{N}\sum_{i=1}^N Z_t(i)$, then
$\bar Z_t\xrightarrow[t\to\infty]{a.s.}(d_1+d_2)/N$. Thus Corollary \ref{cor:Z_limit} is immediate.

 \subsubsection{Proof of Theorem \ref{thm:randomlimit}}
Note that we have already proved Theorem \ref{thm:randomlimit} at the end of Section~\ref{sec:convergence_random}. In this section, we prove all the auxilliary results, namely Proposition~\ref{prop:Var} and Theorem~\ref{thm:QM1}, used in the proof of Theorem \ref{thm:randomlimit}. We begin with the following spectral property of the drift function.

\begin{lemma}\label{lem:eigenvector}
    Define $\hat{v}={\bf 1}_{P_1 \cup D_2}-{\bf 1}_{P_2 \cup D_1}$. Then the real left eigenspace of $Q\widetilde{A}-I$ corresponding to the eigenvalue $0$ is spanned by
$\hat v$. In particular, if such a real left eigenvector $v$ is normalized so that
$\max_i |v_i|=1$, then $v=\hat v$ or $v=-\hat v$.
\end{lemma}

\begin{proof}
Let $v$ be a real left eigenvector of $Q\widetilde{A}-I$ corresponding to the eigenvalue $0$. Since the matrix $Q\widetilde{A}-I$ is real and the eigenvalue is real, this entails no loss of generality. Normalize $v$ so that $\max_i |v_i|=1$. Choose $j$ such that $|v_j|=1$, and replacing $v$ by $-v$ if necessary, assume that $v_j=1$.

Assume first that $j \in P_1$. For this node, the total reinforcement received from its in-neighbors comes from nodes in $P_1 \cup D_1$. Hence
\begin{equation*}
\sum_{i \in P_1} \mu_{ij}
+
\sum_{i \in D_1} \mu_{ij} = 1
\end{equation*}
On the other hand, the $j$-th component of the left eigenvector equation
$vQ\widetilde{A}=v$ gives
\begin{equation*}
\sum_{i \in P_1} v_i \mu_{ij}
-
\sum_{i \in D_1} v_i \mu_{ij}=1.
\end{equation*}
Comparing the two expressions  and using $|v_i|\leq 1$, the equality can hold only if, for every $i$ such that $\mu_{ij}>0$,
\begin{itemize}
    \item $v_i=1$ for $i\in P_1$,
    \item $v_i=-1$ for $i\in D_1$.
\end{itemize}

Now fix any $k\in P_1$. We may choose a path from $k$ to $j$ whose block pattern is one of the following:
\begin{itemize}
    \item $k \rightsquigarrow_{P_1} j$,
    \item $k \rightsquigarrow_{P_1} D_2 \rightarrow D_1 \rightarrow P_1 \rightsquigarrow_{P_1} j$,
    \item $k \rightsquigarrow_{P_1} D_2 \rightarrow P_2 \rightsquigarrow_{P_2} D_1 \rightarrow P_1 \rightsquigarrow_{P_1} j$.
\end{itemize}
In the first case, repeatedly applying the previous argument backwards along the path gives $v_k=1$. In the second case, the path contains a node $k'\in D_1$ with $v_{k'}=-1$. For such a node, the eigenvector equation gives
\begin{equation*}
v_{k'}
=
\sum_{i \in P_2} v_i \mu_{ik'}
-
\sum_{i \in D_2} v_i \mu_{ik'}.
\end{equation*}
Using the same comparison argument and backtracking along the path gives $v_k=1$. The third case is identical, with the same argument applied through the intermediate segment in $P_2$. Therefore $v_k=1$ for every $k\in P_1$. The same argument applied to the other blocks gives
\[
v_i=1 \quad \text{for all } i\in P_1\cup D_2,
\qquad
v_i=-1 \quad \text{for all } i\in P_2\cup D_1.
\]
Thus $v=\hat v$. If the initial choice of sign was the opposite, then $v=-\hat v$. If $j$ belongs to any of the other three partitions, the same argument, with the signs changed according to the partition containing $j$, gives the same conclusion. Hence $v=\hat v$ or $v=-\hat v$, and the real left eigenspace is spanned by $\hat v$.
\end{proof}
 The lemma implies that the geometric multiplicity of eigenvalue $0$ is $1$. Combined with assumption {\rm \bf (D)}, we get that the algebraic multiplicity of $0$ is also $1$. We now prove Proposition~\ref{prop:Var}, which gives the rate
of decay of the variance of $Z_tU$ for a suitable matrix $U$.

 \begin{proof}[Proof of Proposition \ref{prop:Var}]
Writing \eqref{noderecurnew} in vector form, we get

\begin{align}\label{eq:EVar}
\IE[\Z_{t+1}\mid\FF_t]
&=
\Z_t\left(1-\frac{1}{t+1}\right)
+\frac{(\Z_tQ+\bone_{\mathcal D})\widetilde A}{t+1}
\nonumber\\
&=
\Z_tK_t+\frac{\bone_{\mathcal D}\widetilde A}{t+1}.
\end{align}
where $K_t:=I-\frac{1}{t+1}\left(I-Q\widetilde{A}\right)$. Let $D_t:=\Z_tU$. Since $Q\widetilde{A}U=UQ\widetilde{A}$, we have $K_tU=UK_t$. Further, since $\mathbf{1}_{\mathcal D}\widetilde{A}U=0$, we get $\IE[D_{t+1}\mid \FF_t]=\Z_tK_tU=\Z_tUK_t=D_tK_t$. Taking expectation and iterating gives
\[
\IE[D_{t+1}]=\Z_0\prod_{i=0}^{t}K_i U.
\]
Using {\rm \bf (D)}, write
$I-Q\widetilde{A}=P\Lambda P^{-1}$, where
$\Lambda=\operatorname{diag}(\nu_1,\ldots,\nu_N)$. Then, $$\prod_{i=0}^{t}K_i
=
P\prod_{i=0}^{t}
\left(I-\frac{1}{i+1}\Lambda\right)P^{-1}.$$ Since $Q$ only changes signs of rows and the column sums of $\widetilde{A}$ are $1$, every eigenvalue $\eta$ of $Q\widetilde{A}$ satisfies $|\eta|\leq 1$. Hence every eigenvalue $\nu=1-\eta$ of $I-Q\widetilde{A}$ satisfies $\operatorname{Re}(\nu)\geq 0$. Let $s$ be such that $\nu_s=0$, and let $e_s$ denote the $s$-th standard row vector. Then,
\[
\prod_{i=0}^{t}
\left(I-\frac{1}{i+1}\Lambda\right)
\longrightarrow e_s^\top e_s.
\]
Moreover, for every $\nu_i\neq 0$,
\[
\prod_{j=0}^{t}
\left|1-\frac{1}{j+1}\nu_i\right|
=
\mathcal{O}\left(t^{-\operatorname{Re}(\nu_i)}\right).
\]
Since the $s$-th row of $P^{-1}$ is a left eigenvector
corresponding to the eigenvalue $0$, it is a multiple of
$\hat v$. Hence $\hat vU=0$ gives $e_sP^{-1}U=0$. Therefore,
$\IE[D_t]=\mathcal{O}\left(t^{-\nu^*}\right)$. Now, since
$\Z_{t+1}
=
\Z_t\left(1-\frac{1}{t+1}\right)
+
\frac{\I_{t+1}\widetilde{A}}{t+1}$,
we have
\begin{eqnarray*}
\operatorname{Var}(\Z_{t+1})
&=&
\IE[\operatorname{Var}(\Z_{t+1}\mid \mathcal F_t)]
+
\operatorname{Var}(\IE[\Z_{t+1}\mid \FF_t])
\\
&=&
\frac{\widetilde{A}^\top}{(t+1)^2}
\IE[\operatorname{Var}(\I_{t+1}\mid \mathcal F_t)]\widetilde{A}
+
K_t^\top\operatorname{Var}(\Z_t)K_t.
\end{eqnarray*}
Define $N_t:=\widetilde{A}^\top \ \IE[\operatorname{Var}(\I_{t+1}\mid \mathcal F_t)]\widetilde{A}$. The entries of $N_t$ are uniformly bounded. Thus
\begin{eqnarray*}\label{Var_decay2}
\operatorname{Var}(D_{t+1})
&=&
U^\top\operatorname{Var}(\Z_{t+1})U
\\
&=&
U^\top
\left[
\sum_{k=0}^{t}
\frac{1}{(k+1)^2}
\left(\prod_{i=0}^{t-k-1}K_{t-i}^\top\right)
N_k
\left(\prod_{i=k+1}^{t}K_i\right)
\right]
U
\\
&=&
U^\top(P^{-1})^\top\Delta_tP^{-1}U,
\end{eqnarray*}
where
\[
\Delta_t:=
\sum_{k=0}^{t}
\frac{1}{(k+1)^2}
\left[
\prod_{i=0}^{t-k-1}
\left(I-\frac{1}{t-i+1}\Lambda\right)
\right]
P^\top N_kP
\prod_{i=k+1}^{t}
\left(I-\frac{1}{i+1}\Lambda\right).
\]
Let $\beta_k(p,q)$ be the $(p,q)$-th entry of $P^\top N_kP$. Since the entries of $N_k$ are uniformly bounded, so are the $\beta_k(p,q)$'s. The $(p,q)$-th entry of $\Delta_t$ is
\[
\Delta_t(p,q)
=
\sum_{k=0}^{t}
\frac{1}{(k+1)^2}
\left[
\prod_{i=0}^{t-k-1}
\left(1-\frac{1}{t-i+1}\nu_p\right)
\right]
\beta_k(p,q)
\prod_{i=k+1}^{t}
\left(1-\frac{1}{i+1}\nu_q\right).
\]
Using Euler's approximation for $(p,q)\neq(s,s)$, there is a constant $C'$ such that
\begin{eqnarray}\label{rates2}
|\Delta_t(p,q)|
&\leq&
C't^{-(\operatorname{Re}(\nu_p)+\operatorname{Re}(\nu_q))}
\sum_{k=1}^{t}
\frac{1}{k^{2-(\operatorname{Re}(\nu_p)+\operatorname{Re}(\nu_q))}}
\nonumber \\
&=&
\begin{cases}
\mathcal{O}\left(t^{-1}\right),
& \text{if } \operatorname{Re}(\nu_p)+\operatorname{Re}(\nu_q)>1,\\
\mathcal{O}\left(\frac{\log t}{t}\right),
& \text{if } \operatorname{Re}(\nu_p)+\operatorname{Re}(\nu_q)=1,\\
\mathcal{O}\left(t^{-(\operatorname{Re}(\nu_p)+\operatorname{Re}(\nu_q))}\right),
& \text{if } 0<\operatorname{Re}(\nu_p)+\operatorname{Re}(\nu_q)<1.
\end{cases}
\end{eqnarray}
Also, $\Delta_t(s,s)$ converges to
$\sum_{k=0}^{\infty}\frac{1}{(k+1)^2}\beta_k(s,s)
\eqqcolon C$. Then, $\Delta_{t} \to C e_{s}^\top e_{s}$ as $t \to \infty$. Using again $e_s^\top e_sP^{-1}U=0_{N\times N}$, we get that
$\operatorname{Var}(D_{t+1})=\mathcal{O}(t^{-\delta})$ for some $\delta>0$.
\end{proof}

To prove Theorem~\ref{thm:QM1} we need the following.
\begin{lemma}\label{lem:limit_Z}
Assume that {\rm \bf (B)} and {\rm \bf (C)} hold, and suppose that $p_1+p_2=d_1+d_2$. Then $\Z_t\tilde J$ converges almost surely. If
$\Z_\infty$ denotes its almost sure limit, then
\[
\Z_\infty(Q\widetilde{A}-I)=-\mathbf 1_{\mathcal D}\widetilde{A}.
\]
\end{lemma}
\begin{proof}
Under the standing assumption {\rm \bf (A2)}, we have
$\bone\widetilde A=\bone$. Multiplying the two identities
in {\rm \bf (B)} on the left by $\bone$ therefore gives
\[
N\mu_{\mathcal P}^{out}=p_1+p_2,
\qquad
N\mu_{\mathcal D}^{out}=d_1+d_2.
\]
Since $p_1+p_2=d_1+d_2$ and these two quantities sum to $N$,
it follows that
\[
\mu_{\mathcal P}^{out}
=\mu_{\mathcal D}^{out}
=\frac12.
\]
In particular,
$\mu_{\mathcal P}^{out}-\mu_{\mathcal D}^{out}=0<1$.
The final part of Theorem~\ref{thm:AS_average} now implies
that all four block sums
$Z_t^{P_1},Z_t^{P_2},Z_t^{D_1},Z_t^{D_2}$
converge almost surely. Consequently,
\[
Y_t:=\Z_t\tilde J
=
\frac{Z_t^{P_1}}{p_1}\bone_{P_1}
+\frac{Z_t^{P_2}}{p_2}\bone_{P_2}
+\frac{Z_t^{D_1}}{d_1}\bone_{D_1}
+\frac{Z_t^{D_2}}{d_2}\bone_{D_2}
\]
converges almost surely. Denote its limit by $\Z_\infty$. Note that $H = Q \tilde{A} - I$.
Taking conditional expectations in \eqref{noderecurnew}
and writing the result in vector form gives
\[
\IE[\Z_{t+1}\mid\FF_t]
=
\Z_t+\frac{1}{t+1}\bigl(\Z_tH+ \bone_{\mathcal D}\widetilde A \bigr).
\]
By Lemma~\ref{lem:commute}, $\tilde J$ commutes with $H$.
Moreover, {\rm \bf (C)} implies that
$b:=\bone_{\mathcal D}\widetilde A$ is constant on each
of the four blocks, so $b\tilde J=b$. Multiplying the
preceding identity on the right by $\tilde J$ yields
\[
\IE[Y_{t+1}\mid\FF_t]
=
Y_t+\frac{1}{t+1}\bigl(Y_tH+\bone_{\mathcal D}\widetilde A\bigr).
\]
Thus,
\[
Y_{t+1}
=
Y_t+\frac{1}{t+1} (Y_tH+\bone_{\mathcal D}\widetilde A)+\Delta R_{t+1},
\qquad
\Delta R_{t+1}
:=
Y_{t+1}-\IE[Y_{t+1}\mid\FF_t].
\]

The sequence $\{\Delta R_{t+1}\}$ is a martingale
difference sequence. More explicitly, writing
$\I_{t+1}$ for the row vector of sampling indicators,
\eqref{noderecurnew} gives
$\Delta R_{t+1} = \frac{1}{t+1}
\left(\I_{t+1}-\IE[\I_{t+1}\mid\FF_t]\right)\widetilde A\tilde J$. Since the indicators are bounded and the matrices are fixed, there is a deterministic constant $C<\infty$
such that $\IE\!\left[\|\Delta R_{t+1}\|_2^2\mid\FF_t\right]
\leq \frac{C}{(t+1)^2}$. Hence the partial sums of $\sum_{t=0}^{\infty}\Delta R_{t+1}$ form an $\mathcal L^2$-bounded martingale and converge almost surely. Summing the recursion for $Y_t$ gives
\[
\sum_{t=0}^{n}\frac{1}{t+1}(Y_tH+\bone_{\mathcal D}\widetilde A)
=
Y_{n+1}-Y_0-\sum_{t=0}^{n}\Delta R_{t+1}.
\]
The right-hand side converges almost surely. Therefore
the vector series on the left converges almost surely.
On the other hand,
\[
Y_tH+\bone_{\mathcal D}\widetilde A \longrightarrow\Z_\infty H+ \bone_{\mathcal D}\widetilde A
\qquad\text{almost surely}.
\]
If any coordinate of this limiting vector were nonzero,
the corresponding coordinate of $Y_tH+\bone_{\mathcal D}\widetilde A$ would
eventually have a fixed sign and absolute value bounded
away from zero. Since
$\sum_{t=0}^{\infty}(t+1)^{-1}=\infty$, the corresponding
coordinate series would then diverge, a contradiction.
Thus $\Z_\infty H+\bone_{\mathcal D}\widetilde A=0$, which is precisely
\[
\Z_\infty(Q\widetilde A-I)
=-\bone_{\mathcal D}\widetilde A.
\]\end{proof}

The final result needed to Prove the main partial convergence result stated in Theorem~\ref{thm:randomlimit} is Theorem \ref{thm:QM1}, the proof of which is given below.
\begin{proof}[Proof of Theorem \ref{thm:QM1}] We show that $\Z_t$ is a quasi-martingale. By Lemma~\ref{lem:limit_Z}, $\Z_t\tilde J$ converges almost surely. Denote its limit by $\Z_\infty$. Set $Y_t=\Z_t\tilde J$, and $\tilde{\Z}_t=Y_t-\Z_\infty$. From \eqref{eq:EVar}, \[ \IE[\Z_{t+1}\mid\mathcal F_t]-\Z_t = \frac{1}{t+1}\left[\Z_tH+\mathbf 1_{\mathcal D}\widetilde{A}\right]. \] Using $\Z_tH=\Z_t(I-\tilde J)H+Y_tH$ and $\Z_\infty H+\mathbf 1_{\mathcal D}\widetilde A=0$ from Lemma~\ref{lem:limit_Z}, we get \[ \Z_tH+\mathbf 1_{\mathcal D}\widetilde{A} = \Z_t(I-\tilde J)H+\tilde{\Z}_tH. \] Therefore, $\sum_{t=0}^{\infty} \IE[\|\IE[\Z_{t+1}\mid\mathcal F_t]-\Z_t\|] \leq \sum_{t=0}^{\infty} \frac{\IE[\|\Z_t(I-\tilde J)H\|]}{t+1} + \sum_{t=0}^{\infty} \frac{\IE[\|\tilde{\Z}_tH\|]}{t+1}$. We first consider the term $\Z_t(I-\tilde J)H$. We apply Proposition~\ref{prop:Var} with
$H_1=(I-\tilde J)H$. By {\rm \bf (C)}, $\mathbf 1_{\mathcal D}\widetilde{A}$ is constant on the four parts of the partition, so $\mathbf 1_{\mathcal D}\widetilde{A}\tilde J
=\mathbf 1_{\mathcal D}\widetilde{A}$, and hence
$\mathbf 1_{\mathcal D}\widetilde{A}H_1=0$. By Lemma~\ref{lem:commute}, $\tilde J$
commutes with $Q\widetilde A$; since $H=Q\widetilde{A}-I$ also commutes with $Q\widetilde{A}$, we
get $Q\widetilde AH_1=H_1Q\widetilde A$. Finally, since $\hat v$ is constant on the four parts of
the partition, $\hat v\tilde J=\hat v$, and hence $\hat vH_1=0$. Hence Proposition~\ref{prop:Var} gives $\IE[\Z_tH_1]=\mathcal O(t^{-a_1})$ and $\operatorname{Var}(\Z_tH_1)=\mathcal O(t^{-b_1})$ for some $a_1,b_1>0$. Thus $\IE[\|\Z_tH_1\|]=\mathcal O(t^{-c_1})$ for some $c_1>0$. Since $t+1$ grows linearly, \[ \sum_{t=0}^{\infty} \frac{\IE[\|\Z_t(I-\tilde J)H\|]}{t+1}<\infty. \] For the second term, set
$G_t:=Y_tH+\mathbf 1_{\mathcal D}\widetilde{A}$. By Lemma~\ref{lem:limit_Z}, $\Z_\infty H+\mathbf 1_{\mathcal D}\widetilde{A}=0$, and hence
$G_t=(Y_t-\Z_\infty)H=\tilde{\Z}_tH$. From the recursion for $Y_t$ obtained in
Lemma~\ref{lem:limit_Z},
\[
Y_{t+1} = Y_t+\frac{1}{t+1} \left(Y_tH+\mathbf 1_{\mathcal D}\widetilde{A}\right) +\Delta R_{t+1},
\]
where $\Delta R_{t+1}$ is a martingale difference and
$\sum_t\IE[\|\Delta R_{t+1}\|^2\mid\FF_t]<\infty$ almost surely. Multiplying the above recursion by $H$
and adding $\mathbf 1_{\mathcal D}\widetilde{A}$ gives
\[
G_{t+1}=G_tK_t+\Delta R_{t+1}H.
\]

Now write $I-Q\widetilde{A}=P\Lambda P^{-1}$, where
$\Lambda=\operatorname{diag}(\nu_1,\ldots,\nu_N)$, and choose $s$ such
that $\nu_s=0$. Since
$H=-P\Lambda P^{-1}$, we have
$HPe_s^\top=-P\Lambda e_s^\top=0$. Therefore
$G_tPe_s^\top=0$ and $\Delta R_{t+1}HPe_s^\top=0$. Thus, if $U_t=G_tP$ and
$\xi_{t+1}=\Delta R_{t+1}HP$, then $U_t(s)=\xi_{t+1}(s)=0$, while for
$i\neq s$,
\[
U_{t+1}(i)
= U_t(i)\left(1-\frac{\nu_i}{t+1}\right) + \xi_{t+1}(i).
\]
Since $\operatorname{Re}(\nu_i)>0$ for $i\neq s$ and
$\IE[\|\xi_{t+1}\|^2\mid\FF_t]=\mathcal O(t^{-2})$, iterating the scalar recursion and using the same Euler-product estimates as in Proposition~\ref{prop:Var} gives
$\IE[\|G_t\|]=\mathcal O(t^{-a_2})$ for some $a_2>0$. Consequently, since
$G_t=\tilde{\Z}_tH$,
\[
\sum_{t=1}^{\infty}
\frac{\IE[\|\tilde{\Z}_tH\|]}{t+1}
=
\sum_{t=1}^{\infty}
\frac{\IE[\|G_t\|]}{t+1}
<\infty.
\]
Combining the two estimates, we get
\[
\sum_{t=0}^{\infty}
\IE[\|\IE[\Z_{t+1}\mid\mathcal F_t]-\Z_t\|]<\infty.
\]
Since $\Z_t$ is bounded, it follows that $\Z_t$ is a bounded quasimartingale. Hence $\Z_t$ admits an almost sure limit.
\end{proof}
\section{Embedding of classical models} \label{sec:App}
This section illustrates how several classes of interacting stochastic systems can be represented within the framework introduced in Section~\ref{sec:model}. In each case, we identify the sampling graph, reinforcement graph, and the corresponding parameters $(W,Q,\widetilde A,C)$.

\subsection{Edge-dependent reinforcement: interacting urns} \label{sec:urns}
Suppose a two-colour urn is placed at each vertex and let $Z_t(i)$ denote the fraction of white balls in urn at vertex $i$ at time $t$. Consider the following reinforcement matrix along the directed edge $j\to i$,
\begin{equation}\label{eq:matrix_edge_reinforcement}
    R_{ji} = \begin{pmatrix}
        \alpha_{ji} & m_{ji}-\alpha_{ji} \\
        m_{ji} - \beta_{ji} & \beta_{ji}
    \end{pmatrix}
\end{equation}
Like in \cite{kaur2025interacting}, we categorize the reinforcement as 
\begin{itemize}
    \item P\'olya type if $R_{ji} = m_{ji} I_2$, where $I_2$ denotes a $2 \times 2$ identity matrix.
    \item Non-P\'olya type if  $0 <\alpha_{ji} + \beta_{ji} < 2m_{ji}.$
\end{itemize}

\begin{proposition} \label{prop:urn}
The normalized interacting urn model satisfies
\eqref{eq:main_recur_vec}, with
$\mu_{ji}=(\alpha_{ji}+\beta_{ji}-m_{ji})/\tilde m_i$
and $c_{ji}=(m_{ji}-\beta_{ji})/\tilde m_i$,
where $\tilde m_i=\sum_{j\in N(i)}m_{ji}$.
The remainder is given in the proof below.
\end{proposition}

\begin{proof}
The evolution of the number of white balls is given by
$$W_{t+1}(i)=W_t(i) + \sum_{j \in N(i)}\{ \alpha_{ji} \I_{t+1}(j) + (m_{ji}-\beta_{ji})(1-\I_{t+1}(j))\},$$
where
$\I_{t+1}(j)\mid\FF_t\sim
\operatorname{Ber}\!\left(
(1-q_j)+(2q_j-1)\sum_{k\in\VV}w_{kj}Z_t(k)
\right)$.
Note that the total number of balls in urn $i$ updates as $T_{t+1}(i)=T_t(i) + \tilde{m}_i$, where $\tilde{m}_i=\sum_{j \in N(i)}m_{ji}$. We get
\begin{eqnarray*}
Z_{t+1}(i)&=
&\frac{Z_t(i)T_t(i)}{T_{t+1}(i)} +\frac{1}{T_{t+1}(i)}\left( \sum_{j \in N(i)}(\alpha_{ji}+\beta_{ji}-m_{ji})\I_{t+1}(j) +   \sum_{j \in N(i)}(m_{ji}-\beta_{ji} )            \right) \\
&=& Z_t(i)\left(1-\frac{\tilde{m}_i}{T_{t+1}(i)} \right)+\frac{\tilde{m}_i}{T_{t+1}(i)}\left( \sum_{j \in N(i)}\mu_{ji}\I_{t+1}(j) +   \sum_{j \in N(i)}\frac{m_{ji}-\beta_{ji}}{\tilde{m}_i} \right)
\end{eqnarray*}
where $\mu_{ji}=\frac{\alpha_{ji}+\beta_{ji}-m_{ji}}{\tilde{m}_i}$. 

Hence,
\begin{align*}
Z_{t+1}(i) &=Z_t(i)\left(1-\frac{\tilde{m}_i}{T_{t+1}(i)} \right)+ \frac{\tilde{m}_i}{T_{t+1}(i)}\Big[ 
\sum_{j \in N(i)} \left(\mu_{ji} \mathbb{E}\left[ \I_{t+1}(j) \mid \mathcal{F}_t \right] + \frac{m_{ji} - \beta_{ji}}{\tilde{m}_i} \right) + \Delta M_{t+1}(i)\Big]\\
&=Z_t(i)\left(1-\frac{1}{t+1} \right)+ \frac{1}{t+1}\Big[ 
\sum_{j \in N(i)} \left(\mu_{ji} \mathbb{E}\left[ \I_{t+1}(j) \mid \mathcal{F}_t \right] + \frac{m_{ji} - \beta_{ji}}{\tilde{m}_i} \right) + \Delta M_{t+1}(i) +\epsilon_{t+1}(i)\Big],
\end{align*}
where  $ \mathbb{E}\left[ \I_{t+1}(j) \mid \mathcal{F}_t \right] =(1 - q_j) + (2q_j - 1) \sum_{k\in\VV} w_{kj} Z_t(k)$, $\Delta M_{t+1}(i)=\sum_{j \in N(i)} \mu_{ji}\big(\I_{t+1}(j)-\mathbb{E}\left[ \I_{t+1}(j) \mid \mathcal{F}_t \right]\big)$, and
\begin{align*}
\epsilon_{t+1}(i) 
&= \left(\frac{(t+1)\tilde{m}_i}{T_{t+1}(i)}-1\right)\Bigg[ 
\sum_{j \in N(i)} \left(\mu_{ji}\mathbb{E}\left[ \I_{t+1}(j) \mid \mathcal{F}_t \right] + \frac{m_{ji} - \beta_{ji}}{\tilde{m}_i} \right)
 + \Delta M_{t+1}(i) - Z_t(i) \Bigg]. \\
 \end{align*} 
\end{proof}
Define $E^{(n)}_a = \{ (v_1, v_2): \exists v_3, \dots, v_n \text{ such that }  v_1 \to v_2 \to v_3 \to \dots \to v_n \to a \}$. We say that the edges $e_1=(v_1, v_2)$ and $e_n=(v_n, a)$ are connected by a path of length $n$. As a consequence of Theorem~\ref{thm:deterministic_limit}, we get the following convergence result for the fraction of white balls across urns.
\begin{corollary}
   Suppose that the sampling and reinforcement graphs have the
same underlying strongly connected directed graph.
Then $WQ\widetilde A-I$ is invertible provided at least one
of the following conditions holds.
\begin{itemize}
  \item[(i)] There exists $r^* \in [N]$ such that $|2q_{r^*}-1|<1$ and either there exists an odd cycle in the graph or there exists another vertex $r^\prime$ such that $|2q_{r^\prime} - 1|<1$ and $r^* \rightsquigarrow r^\prime$ (or $r^\prime \rightsquigarrow r^*$) via a path of odd length.
  \item[(ii)] There exists a non-P\'olya edge $e_1=(u_1, v_1)$ in the graph and either the graph has an odd cycle or there exists another non-P\'olya edge $e_2 = (u_2, v_2)$ such that $u_1 \rightsquigarrow u_2$ (or $u_2 \rightsquigarrow u_1$) via a path of odd length.
  \item[(iii)] All edges are P\'olya-type and $q_v\in\{0,1\}$ for every $v\in\VV$. $\GG$ contains an odd cycle and does not admit a partition under the GPA.
\end{itemize}
\end{corollary}

\begin{proof}
    Observe that a vertex $v$ is subcritical if and only if one of the incoming edges at $v$ is of non-P\'olya type. The proof then follows from the proof of Theorem~\ref{thm:invertibility_1}, \ref{thm:GPA} and Proposition~\ref{prop:urn}.  
\end{proof}

Remark~\ref{rem:Weak} applies for weakly connected graphs. Proposition~\ref{prop:evencycle} shows how the placement of subcritical vertices (in this case P\'olya and non-P\'olya edges), together with the identities of the vertices as preferential, de-preferential, or neither, can change whether $WQ\widetilde A-I$ is invertible. We illustrate this through an example on the $4$ vertices. 
\begin{example}[A 4-cycle]\label{ex:4cycle}
Let $k=4$. Then,  the two-step components are $\VV_1=\{1,3\}$ and $\VV_2=\{2,4\}$, and
\[
\det(WQ\widetilde A-I)
=
\left(1-(2q_1-1)(2q_3-1)\mu_{12}\mu_{34}\right)
\left(1-(2q_2-1)(2q_4-1)\mu_{23}\mu_{41}\right).
\]
In the figures below, a blue vertex is preferential, a red
vertex is de-preferential, an edge labelled P is of P\'olya type and an edge labelled NP is of non-P\'olya type. A white vertex carries no colour-coded assumption; its status is only what is specified in the corresponding case.

\begin{itemize}
\item[(i)] Let $q_2=q_4=1$ and $\mu_{23}=\mu_{41}=1$.
\begin{figure}[H]
\centering
\begin{tikzpicture}[->, thick, node distance=2cm]
    \node[circle, draw, minimum size=1cm] (1) at (0,0) {1};
    \node[circle, draw, minimum size=1cm, fill=blue!70] (2) at (0,2) {2};
    \node[circle, draw, minimum size=1cm] (3) at (2,2) {3};
    \node[circle, draw, minimum size=1cm, fill=blue!70] (4) at (2,0) {4};

    \path (1) edge node[left] {NP} (2);
    \path (2) edge node[above] {P} (3);
    \path (3) edge node[right] {NP} (4);
    \path (4) edge node[below] {P} (1);
\end{tikzpicture}
\caption{A 4-cycle with $2,4\in\PP$ and P\'olya edges $2\to3$ and $4\to1$.}
\label{fig:cycle4_case_i}
\end{figure}
In this case, the second determinant factor is zero, and hence
$WQ\widetilde A-I$ is singular. This conclusion only uses the component
$\VV_2=\{2,4\}$; no assumption about the component $\VV_1=\{1,3\}$ is needed.

\item[(ii)] Let $q_2=q_4=1$, assume  $q_1,q_3\notin\{0,1\}$ (that is,  $1$ and $3$ are not in $\PP\cup\DD$), and also let
$\mu_{12}=\mu_{34}=1$. 
\begin{figure}[H]
\centering
\begin{tikzpicture}[->, thick, node distance=2cm]
    \node[circle, draw, minimum size=1cm] (1) at (0,0) {1};
    \node[circle, draw, minimum size=1cm, fill=blue!70] (2) at (0,2) {2};
    \node[circle, draw, minimum size=1cm] (3) at (2,2) {3};
    \node[circle, draw, minimum size=1cm, fill=blue!70] (4) at (2,0) {4};

    \path (1) edge node[left] {P} (2);
    \path (2) edge node[above] {NP} (3);
    \path (3) edge node[right] {P} (4);
    \path (4) edge node[below] {NP} (1);
\end{tikzpicture}
\caption{A 4-cycle with $2,4\in\PP$ and P\'olya edges $1\to2$ and $3\to4$.}
\label{fig:cycle4_case_ii}
\end{figure}
Both factors are non-zero under these assumptions, and the matrix $WQ\widetilde A-I$ is invertible.

\item[(iii)] Let $q_1=0$, $q_2=q_3=q_4=1$, and
$\mu_{23}=\mu_{41}=1$. 
\begin{figure}[H]
\centering
\begin{tikzpicture}[->, thick, node distance=2cm]
    \node[circle, draw, minimum size=1cm, fill=red!75] (1) at (0,0) {1};
    \node[circle, draw, minimum size=1cm, fill=blue!70] (2) at (0,2) {2};
    \node[circle, draw, minimum size=1cm, fill=blue!70] (3) at (2,2) {3};
    \node[circle, draw, minimum size=1cm, fill=blue!70] (4) at (2,0) {4};

    \path (1) edge node[left] {NP} (2);
    \path (2) edge node[above] {P} (3);
    \path (3) edge node[right] {NP} (4);
    \path (4) edge node[below] {P} (1);
\end{tikzpicture}
\caption{A 4-cycle with $2,3,4\in\PP$, $1\in\DD$, and P\'olya edges $2\to3$ and $4\to1$.}
\label{fig:cycle4_case_iii}
\end{figure}
The second determinant factor is zero, and hence $WQ\widetilde A-I$ is
singular.

\item[(iv)] Let $q_1=0$, $q_2=q_3=q_4=1$, and
$\mu_{12}=\mu_{34}=1$. 
\begin{figure}[H]
\centering
\begin{tikzpicture}[->, thick, node distance=2cm]
    \node[circle, draw, minimum size=1cm, fill=red!75] (1) at (0,0) {1};
    \node[circle, draw, minimum size=1cm, fill=blue!70] (2) at (0,2) {2};
    \node[circle, draw, minimum size=1cm, fill=blue!70] (3) at (2,2) {3};
    \node[circle, draw, minimum size=1cm, fill=blue!70] (4) at (2,0) {4};

    \path (1) edge node[left] {P} (2);
    \path (2) edge node[above] {NP} (3);
    \path (3) edge node[right] {P} (4);
    \path (4) edge node[below] {NP} (1);
\end{tikzpicture}
\caption{A 4-cycle with $2,3,4\in\PP$, $1\in\DD$, and P\'olya edges $1\to2$ and $3\to4$.}
\label{fig:cycle4_case_iv}
\end{figure}
$WQ\widetilde A-I$ is
invertible in this case.
\end{itemize}
\end{example}

In the random regime case, some additional assumptions are required. We list them below. Further, we only discuss the node-based reinforcement case as a generalization of \cite{kaur2025interacting}, since the general edge-based model is already discussed in section~\ref{sec:convergence_random} and the results here can easily be extended. 

Since the reinforcement is node-based and of P\'olya type, the reinforcement
weight depends only on the source vertex. We therefore write
$M=\operatorname{diag}(m_1,m_2,\ldots,m_N)$. We first list some conditions required to prove the results in this section.
\begin{itemize}
  \item[(B1)] {\bf Uniform incoming reinforcement (depending on the vertex's sampling type):} there exist
  $\bar{m}^{in}_{\mathcal{P}}\geq 0$ and $\bar{m}^{in}_{\mathcal{D}}\geq 0$ such that
  \[
  \bone MA=\bar{m}^{in}_{\mathcal{P}}\bone_{\mathcal{P}}
  +\bar{m}^{in}_{\mathcal{D}}\bone_{\mathcal{D}}.
  \]
  That is,
  $\sum_{j\in N(i)}m_j=\bar{m}^{in}_{\mathcal{P}}$ for all $i\in\mathcal{P}$, and
  $\sum_{j\in N(i)}m_j=\bar{m}^{in}_{\mathcal{D}}$ for all $i\in\mathcal{D}$.

  \item[(B2)] {\bf Uniform outgoing reinforcement:} there exist
  $m^{out}_{\mathcal{P}}\geq 0$ and $m^{out}_{\mathcal{D}}\geq 0$ such that
  \[
  MA\bone_{\mathcal{P}}^{\top}=m^{out}_{\mathcal{P}}\bone^{\top}
  \quad \text{and} \quad
  MA\bone_{\mathcal{D}}^{\top}=m^{out}_{\mathcal{D}}\bone^{\top}.
  \]
  Equivalently, for every $i\in[N]$,
  \[
  m^{out}_{\mathcal{P}}=d^{out}_{i,\mathcal{P}}m_i
  \quad \text{and} \quad
  m^{out}_{\mathcal{D}}=d^{out}_{i,\mathcal{D}}m_i.
  \]

  \item[(B3)] {\bf Balance condition:}
  $\bone MA=\bar{m}^{in}\bone$. Note that (B3) implies (B1), with
  $\bar{m}^{in}_{\mathcal{P}}=\bar{m}^{in}_{\mathcal{D}}=\bar{m}^{in}$. With this condition
  and the specific structure of the adjacency matrix, we get
  \begin{eqnarray}\label{1MA}
  \bone_{P_1\cup D_1}MA
  &=&
  \bar{m}^{in}\bone_{P_1\cup D_2},
  \nonumber\\
  \bone_{P_2\cup D_2}MA
  &=&
  \bar{m}^{in}\bone_{P_2\cup D_1}.
  \end{eqnarray}
\end{itemize}
For the rest of this section, we assume that the initial total number of balls
is the same in every urn. Together with (B3), this implies that all urns have
the same total number of balls at each time $t$. If
${\bf T}_t=\operatorname{diag}(T_t(1),\ldots,T_t(N))$, then
${\bf T}_t=T_t I$ for some scalar $T_t$. 

\

For the random-regime results below, take $W=I$ and
$q_i\in\{0,1\}$, and suppose that the strongly connected
reinforcement graph admits the one-step GPA partition
\eqref{adjacency-partition}. Let $T_0>0$ be the common
initial total number of balls. Under (B1), use
\[
\widetilde A_{ji}=\frac{m_jA_{ji}}{\tilde m_i},
\qquad
\tilde m_i=
\begin{cases}
\bar m_{\mathcal P}^{in},&i\in\mathcal P,\\
\bar m_{\mathcal D}^{in},&i\in\mathcal D.
\end{cases}
\]
Condition (B2) then gives condition (B), with
\[
\mu_{\mathcal P}^{out}
=\frac{m_{\mathcal P}^{out}}{\bar m_{\mathcal P}^{in}},
\qquad
\mu_{\mathcal D}^{out}
=\frac{m_{\mathcal D}^{out}}{\bar m_{\mathcal D}^{in}}.
\]
The exact urn step size at vertex $i$ is
\[
a_t(i)=
\frac{\tilde m_i}{T_0+(t+1)\tilde m_i}
=\frac1{t+1}+O(t^{-2}).
\]
Since the reinforcement and state variables are bounded,
replacing $a_t(i)$ by $1/(t+1)$ introduces an
$O(t^{-2})$ error in the update. The stable stochastic
approximation arguments for the deterministic aggregate
limits therefore apply with this summable error. Under (B3), all the step sizes coincide:
\[
a_t=\frac1{t+1+T_0/\bar m^{in}}.
\]
Consequently, under (B2), the exact conditional-drift
cancellation in Lemma~\ref{lem:GPA_structure} remains valid, provided $d_1=d_2$. We obtain the following.
\begin{theorem}\label{thm:AS_average_URNS}
Assume (B1) and (B2) hold and
$\frac{m^{out}_{\mathcal P}}{\bar{m}^{in}_{\mathcal P}}
-\frac{m^{out}_{\mathcal D}}{\bar{m}^{in}_{\mathcal D}}<1$. Then, as
$t \to \infty$,
$$
\begin{pmatrix} Z_t^{\mathcal P} \\  Z_t^{\mathcal D} \end{pmatrix}
\xrightarrow{a.s.}
\frac{d_1+d_2}
{\frac{m^{out}_{\mathcal D}}{\bar{m}^{in}_{\mathcal D}}
-\frac{m^{out}_{\mathcal P}}{\bar{m}^{in}_{\mathcal P}}+1}
\begin{pmatrix}
\frac{m^{out}_{\mathcal P}}{\bar{m}^{in}_{\mathcal P}}\\
\frac{m^{out}_{\mathcal D}}{\bar{m}^{in}_{\mathcal D}}
\end{pmatrix}.
$$
In addition, if
$\frac{m^{out}_{\mathcal P}}{\bar{m}^{in}_{\mathcal P}}
=\frac{m^{out}_{\mathcal D}}{\bar{m}^{in}_{\mathcal D}}=r$ (equal-ratio condition), then
$(Z_t^{P_1}+Z_t^{D_1}, Z_t^{P_1}+Z_t^{P_2}, Z_t^{D_1}+Z_t^{D_2})$
and
$(Z_t^{P_2}+Z_t^{D_2}, Z_t^{P_1}+Z_t^{P_2}, Z_t^{D_1}+Z_t^{D_2})$
converge almost surely to $r(d_1+d_2)(1,1,1)$. If the above equal-ratio
condition holds and (B3) holds, then
$Z_t^{P_1}, Z_t^{P_2}, Z_t^{D_1}$ and $Z_t^{D_2}$ admit almost sure limits. Furthermore, if
$m^{out}_{\mathcal P}=m^{out}_{\mathcal D}=\bar m^{in}/2$, then as
$t\to\infty$, $\bar Z_t\xrightarrow{a.s.}1/2$, provided
$p_1+p_2=d_1+d_2$. Moreover, if $p_1=p_2=d_1=d_2$, then, as $t\to\infty$,
\[
(\bar Z_t^{P_1}+\bar Z_t^{D_1},
\bar Z_t^{P_1}+\bar Z_t^{P_2},
\bar Z_t^{D_1}+\bar Z_t^{D_2},
\bar Z_t^{P_2}+\bar Z_t^{D_2})
\xrightarrow{a.s.}
(1,1,1,1).
\]
\end{theorem}
The proof follows from Theorem~\ref{thm:AS_average} and Corollary~\ref{cor:Z_limit}. In addition to this, we can show that under suitable version of conditions (C) and (D), the urns synchronize on each partition. In particular, as $t \to \infty$, $Z_t(i) \to Z_U$ almost surely for all $i \in U$ where $U \in \{ P_1, P_2, D_1, D_2 \}$. We have the following.
\begin{theorem}[Almost Sure Convergence of $\Z_t$]\label{thm:QM_URNS} Suppose (B1)-(B3), (C) and (D) hold and $p_1+p_2=d_1+d_2$. Then $\Z_t$ has an almost sure limit and $\Z_t(I-\tilde{J}) \to{\bf 0}$ almost surely as $t\to\infty$. \end{theorem}


\subsection{Memory-based reinforcement}\label{sec:memory_urns}
An interesting application of this interacting urn process is to study urns with memory. 
    Suppose that each urn in the graph is of non-P\'olya type. Let \(T_t(i)=W_t(i)+B_t(i)\) and \(Z_t(i)=W_t(i)/T_t(i)\). Let \(\{\I_{t+1}(i)\}_{i\in \VV}\) denote the indicators of the events that a white ball is drawn from urn \(i\) at time \(t+1\). Thus, an urn at node $i$ is updated as follows.
\begin{equation} \label{update_eq}
    W_{t+1}(i)=W_t(i) + \sum_{j \in N(i)}\{ \alpha_{j} \I_{t+1}(j) + (m_{j}-\beta_{j})(1-\I_{t+1}(j))\},
\end{equation}
where the conditional distribution (given the history up to time $t$) of $\I_{t+1}(j)$ is given by
$$ \I_{t+1}(j) = \begin{cases} 
1 & \text{with probability } \eta_j Z_t(j) + (1-\eta_j)\mathfrak{M}_t(j) \\
0 & \text{otherwise}.
\end{cases}$$
Here,
$\mathfrak M_t(j)=\frac{1}{t}\sum_{s=1}^t \I_s(j)$ denotes the empirical
fraction of white draws from urn $j$ up to time $t$, and
$(\eta_j)_{j\in \VV }\in[0,1]^\mathcal{V}$ is the family of node-specific memory
parameters. The parameter $\eta_j$ controls the relative weight assigned
to the current composition $Z_t(j)$ and the empirical memory
$\mathfrak M_t(j)$. We refer to $\mathfrak M_t(j)$ as the memory of urn $j$. Only urns with $\eta_j < 1$ tap into their memory. To represent this memory-based model in the notation of the general framework below, we enlarge the vertex set by adding one memory node for each urn that uses its memory. Let $
S=\{j\in\VV:\eta_j<1\}$,
and define
$
\mathfrak V=\VV\cup\{\mathfrak m(j):j\in S\}.$
On this enlarged vertex set, there are two relevant directed graphs: a sampling graph and a reinforcement graph. The sampling graph is
$\mathfrak G_W=(\mathfrak V,\mathfrak E_W)$,
where
$
\mathfrak E_W
=
\{(j,j):j\in\VV\}
\cup
\{(\mathfrak m(j),j):j\in S\}.$

The edge \(\mathfrak m(j)\to j\) in the sampling graph indicates that the composition of the memory urn enters the sampling probability of urn \(j\). The reinforcement graph is $\mathfrak G_R=(\mathfrak V,\mathfrak E_R)$,
where $\mathfrak E_R
=
\mathcal E
\cup
\{(j,\mathfrak m(j)):j\in S\}$. The edge \(j\to\mathfrak m(j)\) in the reinforcement graph records the colour drawn from urn \(j\) in the corresponding memory urn. For instance, consider the graph in Figure~\ref{fig:urns_with_memory}. Suppose that the urns at \(2,3,4\) use their memories. The original graph, the corresponding sampling graph, and the corresponding reinforcement graph are shown below. The red vertices denote the memory urns.

The urn at node \(\mathfrak m(j)\) records the sampling history of urn \(j\) as follows. At time \(0\), the urn at \(\mathfrak m(j)\) is empty. At time \(t>0\), a white or black ball is added to \(\mathfrak m(j)\) according to whether a white or black ball is drawn from urn \(j\). Then \(\mathfrak M_t(j)\) is exactly the fraction of white balls in the urn \(\mathfrak m(j)\) at time \(t\). Thus, the edge \(j\to\mathfrak m(j)\) belongs to the reinforcement graph and carries a P\'olya-type reinforcement matrix, while the edge \(\mathfrak m(j)\to j\) belongs to the sampling graph and indicates that the composition of the memory urn enters the sampling probability of urn \(j\). Since the original urn \(j\) reinforces its ordinary out-neighbours according to the non-P\'olya matrix \(R_j\), while it reinforces its memory node \(\mathfrak m(j)\) according to a P\'olya-type matrix, the reinforcement rule cannot be attached only to the source node. This motivates edge-based reinforcement.

Let $S=\{j\in\VV:\eta_j<1\}$
be the set of urns that tap into their memory, and let
$\mathfrak V=\VV\cup\{\mathfrak m(j):j\in S\}$
be the enlarged vertex set. We index the vertices so that the original urns are indexed first, followed by the auxiliary memory urns. Thus all matrices
in the enlarged system have dimension \((|\VV|+|S|)\times(|\VV|+|S|)\).

In this example, \(Q=I\), since the sampling is preferential with respect to
the enlarged weighted sampling neighbourhood. The sampling matrix
\(W=(w_{ij})_{i,j\in\mathfrak V}\) is defined as follows:
\[
w_{ij}=
\begin{cases}
\eta_j, & \text{if } i=j,\ j\in\VV,\\
1-\eta_j, & \text{if } i=\mathfrak m(j),\ j\in S,\\
1, & \text{if } i=j,\ j\in\mathfrak V\setminus\VV,\\
0, & \text{otherwise.}
\end{cases}
\]

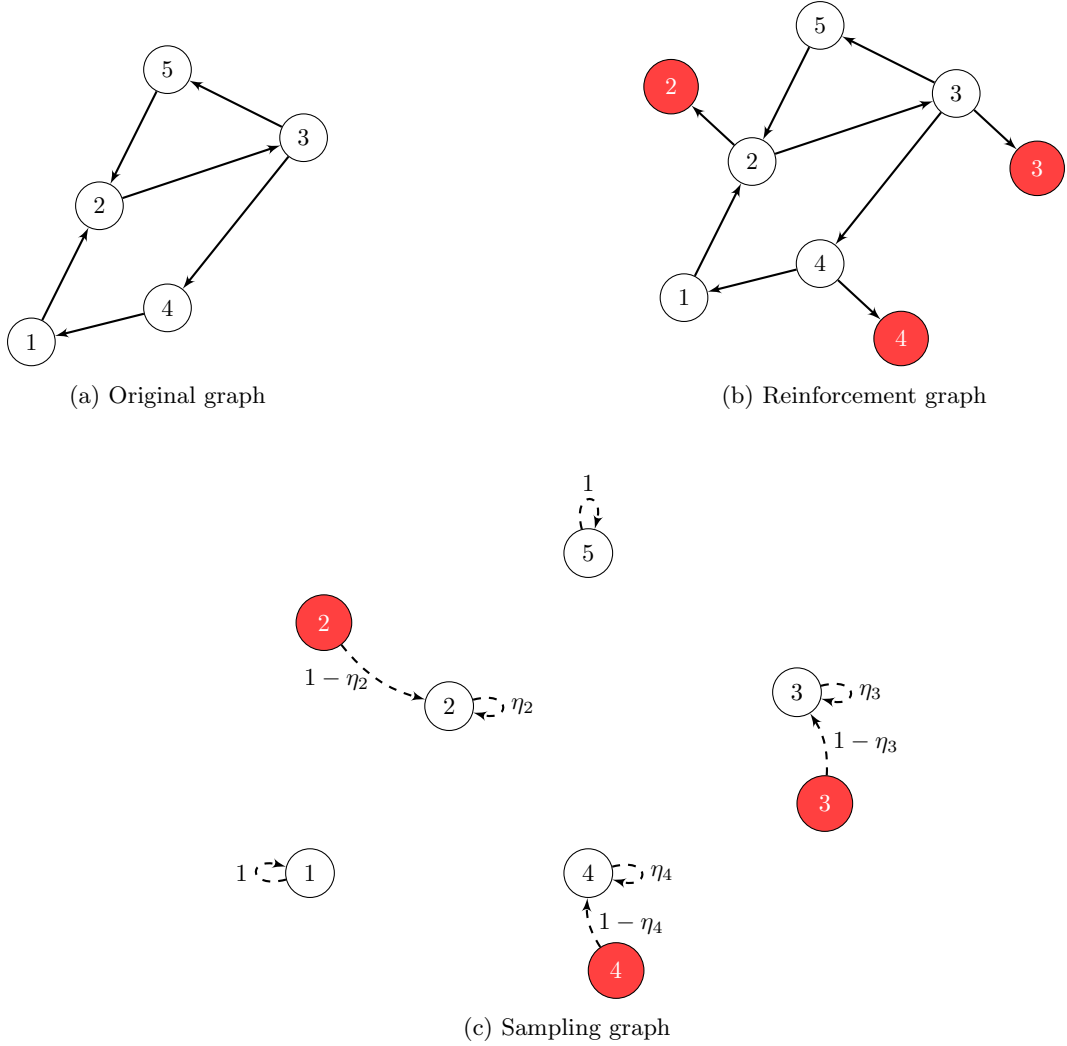
\begin{figure}[H]
\centering
\tikzset{
main/.style={draw, circle, minimum size=0.7cm, inner sep=1pt},
mem/.style={draw, circle, minimum size=0.8cm, inner sep=1pt, fill=red!75, text=white},
edge/.style={->, >=latex', thick},
sedge/.style={->, >=latex', thick, dashed},
redge/.style={->, >=latex', thick}
}

\begin{subfigure}[t]{0.36\textwidth}
\centering
\begin{tikzpicture}[scale=0.9, transform shape]

  \node[main] (a1) at (1,3) {1};  
  \node[main] (a2) at (2,5) {2}; 
  \node[main] (a5) at (3,7) {5};  
  \node[main] (a4) at (3,3.5) {4};  
  \node[main] (a3) at (5,6) {3};  

  \draw[edge] (a1) to (a2);
  \draw[edge] (a2) to (a3);
  \draw[edge] (a3) to (a4);
  \draw[edge] (a4) to (a1);
  \draw[edge] (a5) to (a2);
  \draw[edge] (a3) to (a5);

\end{tikzpicture}
\caption{Original graph}
\end{subfigure}
\hfill
\begin{subfigure}[t]{0.54\textwidth}
\centering
\begin{tikzpicture}[scale=0.9, transform shape]

  \node[main] (a1) at (1,3) {1};  
  \node[main] (a2) at (2,5) {2}; 
  \node[main] (a5) at (3,7) {5};  
  \node[main] (a4) at (3,3.5) {4};  
  \node[main] (a3) at (5,6) {3};  

  \node[mem] (m2) at (0.81,6.10) {2};  
  \node[mem] (m3) at (6.19,4.90) {3};  
  \node[mem] (m4) at (4.19,2.40) {4};  

  \draw[redge] (a1) to (a2);
  \draw[redge] (a2) to (a3);
  \draw[redge] (a3) to (a4);
  \draw[redge] (a4) to (a1);
  \draw[redge] (a5) to (a2);
  \draw[redge] (a3) to (a5);

  \draw[redge] (a2) to (m2);
  \draw[redge] (a3) to (m3);
  \draw[redge] (a4) to (m4);

\end{tikzpicture}
\caption{Reinforcement graph}
\end{subfigure}

\vspace{0.7cm}

\begin{subfigure}[t]{0.58\textwidth}
\centering
\begin{tikzpicture}[scale=0.92, transform shape]

  \node[main] (a1) at (1,2.2) {1};  
  \node[main] (a2) at (3,4.6) {2}; 
  \node[main] (a5) at (5,6.8) {5};  
  \node[main] (a4) at (5,2.2) {4};  
  \node[main] (a3) at (8,4.8) {3};  

  \node[mem] (m2) at (1.2,5.8) {2};  
  \node[mem] (m3) at (8.4,3.2) {3};  
  \node[mem] (m4) at (5.4,0.8) {4};  

  \draw[sedge] (a1) edge[loop left, looseness=8] node[left] {$1$} (a1);
  \draw[sedge] (a5) edge[loop above, looseness=8] node[above] {$1$} (a5);

  \draw[sedge] (a2) edge[loop right, looseness=8] node[right] {$\eta_2$} (a2);
  \draw[sedge] (a3) edge[loop right, looseness=8] node[right] {$\eta_3$} (a3);
  \draw[sedge] (a4) edge[loop right, looseness=8] node[right] {$\eta_4$} (a4);

  \draw[sedge] (m2) to[bend right=18] node[left] {$1-\eta_2$} (a2);
  \draw[sedge] (m3) to[bend right=18] node[right] {$1-\eta_3$} (a3);
  \draw[sedge] (m4) to[bend left=18] node[right] {$1-\eta_4$} (a4);

\end{tikzpicture}
\caption{Sampling graph}
\end{subfigure}

\caption[Sampling and reinforcement graphs for the memory example.]{The original graph and the two graphs associated with the memory example. Red vertices denote the memory urns corresponding to urns \(2,3,4\). Dashed arrows denote sampling interactions, and the labels on these dashed arrows are the corresponding edge weights. Solid arrows denote reinforcement interactions. In the sampling graph, a dashed arrow from a memory urn to the corresponding original urn indicates that the original urn uses its memory while sampling; the self-loop at an original urn indicates the weight assigned to its current composition. In the reinforcement graph, a solid arrow from an original urn to its memory urn indicates that the draw from the original urn is recorded in its memory urn. These memory-recording reinforcement edges carry P\'olya-type reinforcement.}
\label{fig:urns_with_memory}

\end{figure}

Thus, for each \(j\in S\), the weighted sampling average used by urn \(j\) is
\(\eta_j Z_t(j)+(1-\eta_j)Z_t(\mathfrak m(j))\). This agrees with the
memory-based sampling rule, since
\(Z_t(\mathfrak m(j))=\mathfrak M_t(j)\). The sampling rule assigned to the auxiliary memory urns themselves is immaterial for the evolution of the original urns, because memory urns do not reinforce
the original urns; they only record past draws and enter the sampling rules of
the corresponding original urns. The self-weight above is therefore only a convenient way to define the sampling matrix on the full enlarged vertex set. The reinforcement graph on the enlarged vertex set contains the original reinforcement edges, together with the memory-recording edges \(j\to\mathfrak m(j)\) for \(j\in S\). Along the original reinforcement edges, we use the non-P\'olya matrices from the original model, while along \(j\to\mathfrak m(j)\) we use the P\'olya-type matrix \(I_2\). There is no reinforcement edge from \(\mathfrak m(j)\) to \(j\); the influence of \(\mathfrak m(j)\) on \(j\) is through the sampling matrix \(W\).

\subsection{Random Walks with Memory}
Inspired by the \lq\lq Urns with Memory" example in section~\ref{sec:memory_urns}, we describe an asymmetric elephant random walk and its urn coupling that naturally fits into the process described in \eqref{eq:main_recur_vec}. Note that random walks with memory are not directly a special case of the model because their natural state variables are signed, and their memory mechanism acts on past increments. However, as we shall see, through an appropriate urn coupling, the random-walk dynamics can be represented by an interacting urn process, which is a special case of our framework. The urn coupling for a classical elephant random walk was introduced in \cite{baur2016elephant}. We illustrate a generalization of that in the example below.

\

\noindent {\bf Asymmetric Elephant Random Walk:}
    Consider a walk $S_n^{\alpha,\beta} = \sum_{i=1}^n X_i$, where $S_0=X_0=0$ and
    $$X_1=\begin{cases} \alpha \ & \text{w.p.}\quad  q \\
    -\beta \ & \text{w.p.} \quad 1-q
    \end{cases}$$
    for some fixed $\alpha, \beta > 0$. For $n > 1$, choose $n^\prime \sim$ Unif$(1, \dots, n-1)$ and define
    $$X_n=\begin{cases} \alpha \frac{X_{n^\prime}}{\vert X_{n^\prime} \vert } \ & \text{w.p.} \quad  p \\[2ex]
    -\beta \frac{X_{n^\prime}}{\vert X_{n^\prime} \vert } \ & \text{w.p.} \quad 1-p
    \end{cases}$$
    The parameter $p$ controls the tendency of the walk to follow its memory.
After a previous step is selected uniformly at random, the new step has the
same sign as the selected step with probability $p$, and the opposite sign
with probability $1-p$. The magnitudes of these two possibilities are allowed
to be different: a memory-following step has size $\alpha$, while a
memory-reversing step has size $\beta$. Note that for $\alpha=\beta=1$ this is the classical ERW and can be coupled to a two-colour urn with balls of two colours (say red and blue) with mean replacement matrix $\begin{pmatrix}
        p & 1-p \\ 1-p & p
    \end{pmatrix}$, and $S_n^{1,1} \sim R_n - B_n$ (the difference between the balls of colour red and colour blue at time $n$). 

    In the above walk, we need to track the step sizes as well as the number of times a positive or a negative step is sampled. Therefore, the coupling consists of two urns $(U_1, U_2)$, each with balls of both red and blue colours. $U_1$ has a self-loop and a directed edge from $U_1$ to $U_2$, whereas $U_2$ has zero out-degree. $U_1$ reinforces itself with the mean replacement matrix $R_1$ and reinforces $U_2$ with the mean replacement matrix $R_2$, where
    $$
R_1 = \begin{pmatrix}
        p & 1-p \\
        1-p & p
    \end{pmatrix}
\quad \text{and} \quad
R_2 = \begin{pmatrix}
        p\alpha & (1-p)\beta \\
        (1-p)\beta & p\alpha
    \end{pmatrix}.
$$

More precisely, the walk and the two-urn process can be constructed on the
same probability space so that, for every $n\geq 1$,
$S_n^{\alpha,\beta}=R_n^{(2)}-B_n^{(2)}$,
where $R_n^{(2)}$ and $B_n^{(2)}$ denote the total red and blue masses in
$U_2$ after $n$ steps. Equivalently, $S_n^{\alpha,\beta}$ has the same
distribution as $R_n^{(2)}-B_n^{(2)}$.

    \begin{figure}[h]
    \centering
\begin{tikzpicture}  [scale=.6,auto=center,every node/.style={circle,fill=blue!20}]
\tikzset{edge/.style = {->,> = latex'}}

  \node (a1) at (1,1) {$U_1$};  
  \node (a2) at (5,1)  {$U_2$}; 
 \draw[edge] (a1) to (a2) ;
\draw[edge] [loop above] (a1) to (a1);  
\end{tikzpicture}  
    \caption{An interacting urn representation of the elephant random walk.}
        \label{fig:urn-coupling}
\end{figure}
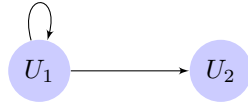

For this urn representation, using the terminology defined above for the edge-based reinforcement model, we have $N(1)=N(2)=\{1\}$, $q_1=q_2=1$, $W=I$, $R_{11}=R_1$ and $R_{12}=R_2$. Hence, 
 $\alpha_{11}=\beta_{11}=p$, $\tilde{m}_1=m_{11}=1$, $\alpha_{12}=\beta_{12}=p\alpha$ and $\tilde{m}_2=m_{12}=p\alpha +(1-p)\beta$. We further compute $\tilde{\mu}_{11}=\frac{\alpha_{11} + \beta_{11}-m_{11}}{\tilde{m}_1}=2p-1$ and $\tilde{\mu}_{12}=\frac{2p\alpha}{p\alpha +(1-p)\beta}-1$. Also, $C=\left(1-p, \frac{(1-p)\beta}{p\alpha +(1-p)\beta}\right)$. Hence, we get that as $t \to \infty$,  \[
\Z_t \;\xrightarrow{a.s.} \;
- \left( 1-p,\; \frac{(1-p)\beta}{p\alpha + (1-p)\beta} \right)
\begin{pmatrix}
2p - 2 
&
\displaystyle \frac{2p\alpha}{p\alpha + (1-p)\beta} - 1
\\[6pt]
0 & -1
\end{pmatrix}^{-1}
=
\left( \frac12,\; \frac12 \right).
\]

Hence,  $S_n^{\alpha,\beta}/n \xrightarrow{} 0$ almost surely as  $n \to \infty$. This can be seen directly by writing a two-dimensional stochastic approximation scheme for $(S_n^{\alpha,\beta}/n, S_n^{1, 1}/n)$.

\section{Discussion and Future Work} \label{sec:conc}
We introduced a framework for interacting reinforced stochastic systems with node-dependent sampling and edge-dependent reinforcement, extending earlier models of interacting urns. Sampling and reinforcement may be governed by distinct directed graphs, and their interaction determines the drift matrix of the associated stochastic approximation scheme. This formulation accommodates heterogeneous reinforcement effects and provides a common setting for studying deterministic limits, random limits, partial and full synchronization, and fluctuations. The applications to interacting urns (extending the framework developed in \cite{kaur2023interacting} and \cite{kaur2025interacting}), urns with memory, and elephant random walks illustrate the range of models captured by the framework.

The model studied in this paper considers interactions determined by the nearest-neighbour structure of the underlying graph. Extending the framework to allow sampling
and reinforcement mechanisms that depend on longer-range neighbourhoods would provide a natural direction for further investigation. In particular, we have the following conjecture.
\begin{conj}
Suppose the sampling and reinforcement graphs coincide, with common underlying graph $G$ of period $d$. If the drift matrix is of the form $W^{m_1}Q\widetilde A^{m_2}-I$, the process converges to a deterministic limit whenever
$\gcd(m_1+m_2,d)=1$.
\end{conj}

The general convergence and fluctuation results allow distinct
sampling and reinforcement graphs. Our graph-theoretic
invertibility criteria are developed under (A1), where the
graphs have the same strongly connected directed edge set.
Extending these structural criteria to arbitrary pairs of
graphs remains an interesting direction for future work. Furthermore, our characterization of the system's behaviour in the absence of odd cycles remains incomplete. In the random-limit regime, the distributional properties of the limiting random vector are largely unknown; establishing whether the limit is non-degenerate is an important open problem. 

The scalar-valued state space considered in this paper can naturally be extended to vector-valued states. In particular, one may consider $\mathbf{Z}_t(i)\in\mathbb{R}^d$, with reinforcement mechanisms that allow different coordinates of an observation to influence one another. The resulting mean-field dynamics are governed by a system of coupled stochastic approximation recursions, with the interaction between network and coordinate structure represented by a block-valued reinforcement operator. The present analysis provides a starting point for such an extension. In the context of interacting urns, this corresponds to interacting multi-colour urn models while in the context of random walks with memory, it leads naturally to interacting walks on $\mathbb{Z}^d$.

\



\appendix

\section{Supplementary Results}  \label{sec:extra}   
We first present the spectral decomposition of the limiting expectation
of $\Z_t\tilde J$.

\begin{proposition}\label{prop:spectral_centering_ZJ}
Work with the random-limit model and the fixed GPA partition
of Section~\ref{sec:convergence_random}.
Assume {\bf (B)}, {\bf (C)}, and {\bf (D)} hold. Write\[
I-Q\widetilde A=P\Lambda P^{-1},
\qquad
\Lambda=\operatorname{diag}(\nu_1,\ldots,\nu_N),
\]
and let $s$ be such that $\nu_s=0$. If $q_k$ denotes
the $k$-th row of $P^{-1}$, then
\[
\lim_{t\to\infty}\IE[\Z_t\tilde J]
=
(\IE[\Z_0\tilde J]P)_s q_s
+
\sum_{k\neq s}
\frac{(\bone_{\mathcal D}\widetilde A P)_k}{\nu_k}\,q_k.
\]
\end{proposition}

\begin{proof}
From~\eqref{eq:EVar}, Lemma~\ref{lem:commute}, and {\rm \bf(C)}, we get
\[
\IE[\Z_{t+1}\tilde J]
=
\IE[\Z_t\tilde J]K_t+\frac1{t+1}\bone_{\mathcal D}\widetilde A,\]
where $K_t=I-\frac{1}{t+1}\left(I-Q\tilde{A} \right)$ is as before. For $i\leq j$, let $\Phi_{i,j}=K_iK_{i+1}\cdots K_j$, and let
$\Phi_{i,j}=I$ if $i>j$. Iterating the preceding recursion gives
\[
\IE[\Z_t\tilde J]
=
\IE[\Z_0\tilde J]\Phi_{0,t-1}
+
\sum_{i=0}^{t-1}\frac{1}{i+1}\bone_{\mathcal D} \tilde{A} \Phi_{i+1,t-1}.
\]
By {\rm \bf (D)}, $\Phi_{i,j}
=
P\prod_{\ell=i}^{j}
\left(I-\frac{1}{\ell+1}\Lambda\right)P^{-1}$. Multiplying on the right by $P$, we obtain, for every $k$,
\begin{eqnarray}\label{eq:centering-coordinate}
(\IE[\Z_t\tilde J]P)_k
=
(\IE[\Z_0\tilde J]P)_k
\prod_{j=0}^{t-1}
\left(1-\frac{1}{j+1}\nu_k\right)
+(\bone_{\mathcal D}\widetilde A P)_k
\sum_{i=0}^{t-1}
\frac{1}{i+1}
\prod_{j=i+1}^{t-1}
\left(1-\frac{1}{j+1}\nu_k\right).
\end{eqnarray}

\noindent For $k=s$, since $\nu_s=0$, \eqref{eq:centering-coordinate} gives
\[
(\IE[\Z_t\tilde J]P)_s
=
(\IE[\Z_0\tilde J]P)_s
+
(\bone_{\mathcal D}\tilde{A}P)_s
\sum_{i=0}^{t-1}\frac{1}{i+1}.
\]
Since $\IE[\Z_t\tilde J]P$ is bounded, we must
have $(\bone_{\mathcal D}\tilde{A}P)_s=0$. Hence $\lim_{t\to\infty}(\IE[\Z_t\tilde J]P)_s
= (\IE[\Z_0\tilde J]P)_s$. For $k\neq s$, the first term in \eqref{eq:centering-coordinate} converges to
$0$. Moreover,
\[
\frac{1}{i+1}
\prod_{j=i+1}^{t-1}
\left(1-\frac{1}{j+1}\nu_k\right)
=
\frac1{\nu_k}
\left[
\prod_{j=i+1}^{t-1}
\left(1-\frac{1}{j+1}\nu_k\right)
-
\prod_{j=i}^{t-1}
\left(1-\frac{1}{j+1}\nu_k\right)
\right].
\]
Therefore,
\[
\sum_{i=0}^{t-1}
\frac1{i+1}
\prod_{j=i+1}^{t-1}
\left(1-\frac{1}{j+1}\nu_k\right)
= \frac{1}{\nu_k}\left[
1- \prod_{j=0}^{t-1}
\left(1-\frac{1}{j+1}\nu_k\right)
\right].
\]
Since $\operatorname{Re}(\nu_k)>0$, the product on the right converges to
$0$. Hence, for $k\neq s$, $\lim_{t\to\infty}(\IE[\Z_t\tilde J]P)_k
=
\frac{(\bone_{\mathcal D}\tilde{A}P)_k}{\nu_k}$. Writing the limiting row vector in the basis given by the rows
$q_k$ of $P^{-1}$ gives
\[
\lim_{t\to\infty}\IE[\Z_t\tilde J]
=
(\IE[\Z_0\tilde J]P)_s q_s
+
\sum_{k\neq s}
\frac{(\bone_{\mathcal D}\tilde{A}P)_k}{\nu_k}\,q_k.
\]
\end{proof}
Observe that for the interacting urn case discussed in section~\ref{sec:urns}, we get
\[
\lim_{t\to\infty}\IE[\Z_t\tilde J]
=
(\IE[\Z_0\tilde J]P)_s q_s
+
\sum_{k\neq s}
\frac{(\bone_{\mathcal D}MAP)_k}{\bar m^{in}\nu_k}\,q_k.
\]
In the next section, we present a special case of the random regime and give an alternative proof of the partial synchronization result. 

\subsection{Alternating Even Cycles}

We now turn to a special class of random-limit examples. Consider directed cycles whose length is divisible by four
and in which preferential and de-preferential vertices alternate. Such graphs admit a GPA partition, and hence belong to the random-limit regime. Unlike the general theorem, however, the cycle structure allows us to identify the limiting block averages and prove synchronization without imposing
the additional assumptions {\rm \bf(C)}, and {\rm \bf (D)}. While the proofs work for the general interacting stochastic processes introduced in this paper, we limit our discussion to interacting urns because we only consider node-based reinforcement here.

Throughout this subsection, we consider the node-based
P\'olya urn model of Section~\ref{sec:urns}, with $W=I$,
on a directed cycle with $N=4n$ vertices.
Preferential and de-preferential vertices alternate.
The initial urn compositions are deterministic, with the
same positive total number of balls in each urn.
The balance condition (B3) is imposed in the theorems below.

\begin{figure}[H]
\centering
\begin{tikzpicture}[->, thick, node distance=2cm]

    \node[circle, draw, minimum size=1cm, fill=red!70] (3) at (0,0) {3};
    \node[circle, draw, minimum size=1cm, fill=blue!70] (1) at (0,2) {1};
    \node[circle, draw, minimum size=1cm, fill=red!70] (4) at (2,2) {4};
    \node[circle, draw, minimum size=1cm, fill=blue!70] (2) at (2,0) {2};

    \path (1) edge (4);
    \path (2) edge (3);
    \path (3) edge (1);
    \path (4) edge (2);

\end{tikzpicture}
\caption{A cycle graph on four vertices with $1,2\in\mathcal P$ and
$3,4\in\mathcal D$.}
\label{fig:alt-cycle-4}
\end{figure}
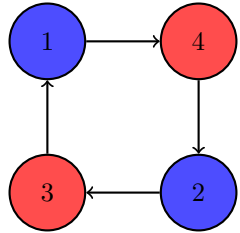

More generally, we consider directed cycles with \(N=4n\) vertices, labelled so that the four parts of the GPA partition satisfy \(p_1=p_2=d_1=d_2=N/4=n\). Since each vertex in a directed cycle has exactly one reinforcement in-neighbour, (B3) forces all reinforcement weights to be equal, and hence \(M=\bar m^{in}I\). According to the labelling of the vertices fixed in the previous section, we
have


\[
\begin{aligned}
P_1&=\{1,2,\ldots,n\},
&\qquad
P_2&=\{n+1,\ldots,2n\},\\
D_1&=\{2n+1,\ldots,3n\},
&\qquad
D_2&=\{3n+1,\ldots,4n\}.
\end{aligned}
\]
For $i\in[N]$, define $\hat i$ to be the parent of $i$, that is, the unique
vertex satisfying $\hat i\to i$. Then
\[
\hat i=
\begin{cases}
3n, & \text{if } i=1,\\
2n+i-1, & \text{if } i>1 \text{ and } i\in P_1,\\
i-n, & \text{if } i\in D_1,\\
2n+i, & \text{if } i\in P_2,\\
i-3n, & \text{if } i\in D_2.
\end{cases}
\]
The following result gives the limiting structure of the block averages.

\begin{theorem}\label{thm:cycle1}
Assume that {\rm \bf (B3)} holds. Then there exists a random variable $\xi_\infty$
taking values in $[0,1]$ such that, as $t\to\infty$,
\[
\bar Z_t^{P_1},\bar Z_t^{D_2}\longrightarrow \xi_\infty
\quad\text{and}\quad
\bar Z_t^{P_2},\bar Z_t^{D_1}\longrightarrow 1-\xi_\infty
\]
almost surely.
\end{theorem}

\begin{proof}
For $i\in[N]$,
\begin{equation}\label{recur_cycle}
Z_{t+1}(i)
=
Z_t(i)
+
\frac{\bar m^{in}}{T_{t+1}}
\begin{cases}
1-Z_t(\hat i)-Z_t(i)+\Delta M_{t+1}(i),
& \text{if } i\in\mathcal P,\\
Z_t(\hat i)-Z_t(i)+\Delta M_{t+1}(i),
& \text{if } i\in\mathcal D,
\end{cases}
\end{equation}
where $\Delta M_{t+1}(i)$ is a bounded martingale difference. Summing over the
four parts of the partition gives
\begin{equation}\label{cycle:P1P2}
Z_{t+1}^{P_1} = Z_t^{P_1} + \frac{\bar m^{in}}{T_{t+1}}
\left(p_1-Z_t^{D_1}-Z_t^{P_1}+\Delta M_{t+1}^{P_1}\right)
\text{ and }
Z_{t+1}^{P_2} = Z_t^{P_2} + \frac{\bar m^{in}}{T_{t+1}}
\left(p_2-Z_t^{D_2}-Z_t^{P_2}+\Delta M_{t+1}^{P_2}\right),
\end{equation}
and,
\begin{equation}\label{cycle:D1D2}
Z_{t+1}^{D_1} = Z_t^{D_1} + \frac{\bar m^{in}}{T_{t+1}} \left(Z_t^{P_2}-Z_t^{D_1}+\Delta M_{t+1}^{D_1}\right) \text{ and }
Z_{t+1}^{D_2} = Z_t^{D_2} +
\frac{\bar m^{in}}{T_{t+1}}
\left(Z_t^{P_1}-Z_t^{D_2}+\Delta M_{t+1}^{D_2}\right).
\end{equation}

From \eqref{cycle:P1P2} and \eqref{cycle:D1D2}, the process $S_t:=Z_t^{P_1}+Z_t^{D_2}-Z_t^{P_2}-Z_t^{D_1}$ is a bounded martingale, and hence admits an almost sure limit, say $S_\infty$. Next, define $V_t = \left( Z_t^{P_1}+Z_t^{D_1},\; Z_t^{P_1}+Z_t^{P_2},\; Z_t^{D_1}+Z_t^{D_2}
\right)$. Using \eqref{cycle:P1P2} and \eqref{cycle:D1D2}, we get
\[
V_{t+1}
=
V_t
+
\frac{\bar m^{in}}{T_{t+1}}
\left[
(p_1,p_1+p_2,0)
+
V_t
\begin{pmatrix}
-2 & 0 & 0\\
1 & -1 & 1\\
0 & -1 & -1
\end{pmatrix}
+
\Delta N_{t+1}
\right],
\]
where $\Delta N_{t+1}$ is a bounded martingale difference. The matrix above
has eigenvalues $-2$ and $-1\pm i$, and the corresponding ODE has the unique
stable equilibrium $(n,n,n)=(N/4,N/4,N/4)$. Hence, by the standard stochastic approximation argument, \(V_t\to(n,n,n)\)
almost surely. Thus \(Z_t^{P_1}+Z_t^{D_1}\), \(Z_t^{P_1}+Z_t^{P_2}\), and
\(Z_t^{D_1}+Z_t^{D_2}\) all converge almost surely to \(n\).

Combining these
three limits with \(S_t\to S_\infty\), we get
\(Z_t^{P_1},Z_t^{D_2}\to(S_\infty+2n)/4\) and
\(Z_t^{P_2},Z_t^{D_1}\to n-(S_\infty+2n)/4\) almost surely. Since each block
has size \(n\), setting \(\xi_\infty=(S_\infty+2n)/(4n)\) gives
\(\bar Z_t^{P_1},\bar Z_t^{D_2}\to\xi_\infty\) and
\(\bar Z_t^{P_2},\bar Z_t^{D_1}\to1-\xi_\infty\) almost surely.
\end{proof}

We now prove a version of Theorem~\ref{thm:randomlimit} for alternating even cycles. Here we give an alternative proof using the convergence theorem for non-negative almost supermartingales (see \cite{RS1971}).

\begin{theorem}\label{thm:cycle2}
Assume that {\rm \bf (B3)} holds. Then there exists a random variable $\xi_\infty$
taking values in $[0,1]$ such that
\[
Z_t(i)\to \xi_\infty
\quad\text{for all } i\in P_1\cup D_2,
\qquad
Z_t(i)\to 1-\xi_\infty
\quad\text{for all } i\in P_2\cup D_1,
\]
almost surely.
\end{theorem}

\begin{proof}[Proof of Theorem~\ref{thm:cycle2}]
Define $D_t=\Z_t(I-\tilde J)$. By Theorem~\ref{thm:cycle1}, it is
enough to show that $D_t\to{\bf 0}$ almost surely. We do this in two steps:
\begin{enumerate}[(i)]
\item $\|D_t\|^2$ admits an almost sure limit.
\item $\IE[\|D_t\|^2]\to0$.
\end{enumerate}
Together, these two statements imply $D_t\to{\bf 0}$ almost surely. From \eqref{recur_cycle}, in vector form, and using
\(\bone_{\mathcal D}A=\bone_{\mathcal P}\), we get
\[
\IE[\Z_{t+1}\mid\FF_t]
=
\Z_t
\left[
\left(1-\frac{\bar m^{in}}{T_{t+1}}\right)I
+
\frac{\bar m^{in}}{T_{t+1}}QA
\right]
+
\frac{\bar m^{in}}{T_{t+1}}\bone_{\mathcal P}.
\]
Since $\bone_{\mathcal P}(I-\tilde J)=0$ and
$QA(I-\tilde J)=(I-\tilde J)QA$, we get
\[
\IE[D_{t+1}\mid\FF_t]
=
D_tK_t,
\qquad
K_t:=I-\frac{\bar m^{in}}{T_{t+1}}(I-QA).
\]

\noindent\textit{Proof of (i).}
Write $D_{t+1}=D_tK_t+\Delta D_{t+1}$, where $\Delta D_{t+1}$ is a martingale
difference. Since the urn increments are bounded and $T_{t+1}$ grows linearly,
there exists a constant $C>0$ such that $\IE[\|\Delta D_{t+1}\|^2\mid\FF_t]\leq \frac{C}{T_{t+1}^2}$. Therefore,
\[
\IE[\|D_{t+1}\|^2\mid\FF_t]
=
\|D_tK_t\|^2+\IE[\|\Delta D_{t+1}\|^2\mid\FF_t].
\]
Since  \(AA^\top=I\), \(QA\) is orthogonal. Using this, we get
\[
I-K_tK_t^\top
=
\frac{\bar m^{in}}{T_{t+1}}
\left(1-\frac{\bar m^{in}}{T_{t+1}}\right)
\{2I-QA-A^\top Q\}
=
\frac{\bar m^{in}}{T_{t+1}}
\left(1-\frac{\bar m^{in}}{T_{t+1}}\right)
(I-QA)(I-A^\top Q).
\]
Thus $I-K_tK_t^\top$ is positive semidefinite, and hence $\IE[\|D_{t+1}\|^2\mid\FF_t]
\leq
\|D_t\|^2+\frac{C}{T_{t+1}^2}$.
Since $\sum_tT_{t+1}^{-2}<\infty$, $\|D_t\|^2$ is an almost supermartingale,
and hence admits an almost sure limit.

\

\noindent\textit{Proof of (ii).} We show that both $\IE[D_t]$ and $\operatorname{Var}(D_t)$ converge to zero.
From the preceding conditional expectation recursion, $\IE[D_{t+1}] =
\IE[D_0]\prod_{i=0}^{t}K_i$. It is easy to verify that the characteristic polynomial of $QA$ is $x^N-1$. Thus $QA$ is diagonalizable and its eigenvalues are simple. Write
$I-QA=P\Lambda P^{-1}$, where
$\Lambda=\operatorname{diag}(\nu_1,\ldots,\nu_N)$.
Let $s$ be such that $\nu_s=0$, and let $e_s$ denote the $s$-th standard row vector. Since the right eigenvector $Pe_s^\top$ corresponding to $\nu_s=0$ is constant on the four parts of the partition, we have
$(I-\tilde J)Pe_s^\top=0$. Hence $D_0Pe_s^\top=0$.
For every $k\neq s$, $\operatorname{Re}(\nu_k)>0$, and Euler's product estimate gives
\[
\prod_{i=0}^{t}
\left(1-\frac{\bar m^{in}}{T_{i+1}}\nu_k\right)
\longrightarrow0.
\]
Therefore $\IE[D_t]\to{\bf 0}$. We now consider the variance.
Let $G_k:=\IE[\operatorname{Var}(\I_{k+1}\mid\FF_k)]$.
The entries of $G_k$ are uniformly bounded. Applying the
projection $I-\tilde J$ to the recursion for $\Z_t$, and then
applying the variance decomposition, gives
\[
\operatorname{Var}(D_{t+1}) = K_t^\top\operatorname{Var}(D_t)K_t
+ \frac{(\bar m^{in})^2}{T_{t+1}^2}
(I-\tilde J)^\top A^\top G_t A(I-\tilde J).\]
Iterating this recursion, and using the fact that the initial configuration is
deterministic, so that \(\operatorname{Var}(D_0)=0\), gives
\[
\operatorname{Var}(D_{t+1}) = \sum_{k=0}^{t} \frac{(\bar m^{in})^2}{T_{k+1}^2} \left(\prod_{i=k+1}^{t}K_i^\top\right)
(I-\tilde J)^\top A^\top G_k A (I-\tilde J) \left(\prod_{i=k+1}^{t}K_i\right),
\]
where an empty product is understood to be the identity matrix. Using the diagonalization of $I-QA$, we have
\[
\prod_{i=k+1}^{t}K_i = P\prod_{i=k+1}^{t} \left(I-\frac{\bar m^{in}}{T_{i+1}}\Lambda\right)P^{-1}.
\]
Thus the corresponding middle term in the eigenbasis has entries of the form
\[
\Delta_t(p,q) = \sum_{k=0}^{t}
\frac{(\bar m^{in})^2}{T_{k+1}^2}
\left[
\prod_{i=k+1}^{t}
\left(1-\frac{\bar m^{in}}{T_{i+1}}\nu_p\right)
\right]
\delta_k(p,q)
\left[
\prod_{i=k+1}^{t}
\left(1-\frac{\bar m^{in}}{T_{i+1}}\nu_q\right)
\right],
\]
where $\delta_k(p,q)$ denotes the $(p,q)$-th entry of
$P^\top (I-\tilde J)^\top A^\top G_k A(I-\tilde J)P$. These entries are uniformly bounded. Moreover,
because $(I-\tilde J)Pe_s^\top=0$, the $s$-th row and the $s$-th column of this middle
matrix are zero. Hence only indices with non-zero eigenvalues contribute.
For $p,q\neq s$, Euler's product estimate gives
\[
|\Delta_t(p,q)|
\leq
C
t^{-(\operatorname{Re}(\nu_p)+\operatorname{Re}(\nu_q))}
\sum_{k=1}^{t}
\frac{1}{k^{2-(\operatorname{Re}(\nu_p)+\operatorname{Re}(\nu_q))}}.
\]
Consequently,
\[
|\Delta_t(p,q)|
=
\begin{cases}
\mathcal O(t^{-1}),&
\text{if }\operatorname{Re}(\nu_p)+\operatorname{Re}(\nu_q)>1,\\
\mathcal O((\log t)/t),&
\text{if }\operatorname{Re}(\nu_p)+\operatorname{Re}(\nu_q)=1,\\
\mathcal O\!\left(t^{-(\operatorname{Re}(\nu_p)+\operatorname{Re}(\nu_q))}\right),&
\text{if }0<\operatorname{Re}(\nu_p)+\operatorname{Re}(\nu_q)<1.
\end{cases}
\]
Since there are only finitely many indices and all non-zero eigenvalues have
positive real parts, this implies $\operatorname{Var}(D_t)\to0$. Therefore
\[
\IE[\|D_t\|^2]
=
\|\IE[D_t]\|^2+\operatorname{Tr}(\operatorname{Var}(D_t))
\longrightarrow0.
\]

By (i), $\|D_t\|^2$ has an almost sure limit. Since $D_t$ is bounded and
$\IE[\|D_t\|^2]\to0$, bounded convergence implies that this almost sure limit
is zero. Hence $D_t\to{\bf 0}$ almost surely.

Finally, $D_t=\Z_t(I-\tilde J)\to{\bf 0}$ means that each coordinate
synchronizes with its block average. Combining this with Theorem~\ref{thm:cycle1} we get that as $t \to \infty$, for every $i\in P_1\cup D_2$, $Z_t(i) \xrightarrow{a.s.} \xi_\infty$ and for every $i\in P_2\cup D_1$,
$Z_t(i)\xrightarrow{a.s.} 1-\xi_\infty$.
\end{proof}
\color{black}


\section{Stochastic Approximation Theory}\label{app:appendix}
\subsection{Almost-sure convergence and ODE approximation}
\begin{definition}
\label{def:sas}
A recursion for a sequence $\{x_t\}_{t\geq0}$ in $\R^d$ of the form
\begin{equation}
\label{eq:recursion}
x_{t+1}
=
x_t+a_t\bigl(h(x_t)+M_{t+1}+\epsilon_{t+1}\bigr),
\qquad t\geq0,
\end{equation}
is called a \emph{stochastic approximation scheme} if it satisfies the following conditions:
\begin{enumerate}
\item[(1)]
The \emph{map} $h\colon\R^d\to\R^d$ is Lipschitz.

\item[(2)]
The \emph{step sizes} $\{a_t\}_{t\geq0}$ satisfy $a_t\in(0,\infty)$ for all $t\geq0$, $\sum_{t\geq0}a_t=\infty$, and $\sum_{t\geq0}a_t^2<\infty$.

\item[(3)]
The sequence $\{M_t\}_{t\geq1}$ is a martingale
difference sequence with respect to $\{\FF_t\}_{t\geq0}$ and, for
some constant $K>0$,
\[
\IE[\|M_{t+1}\|^2\mid\FF_t]
\leq K(1+\|x_t\|^2)
\qquad \prob\text{-a.s.}
\]

\item[(4)]
The \emph{solution} $\{x_t\}_{t\geq0}$ satisfies $\sup_{t\geq0}\|x_t\|<\infty$ $\prob$-a.s.

\item[(5)]
The \emph{errors} $\{\epsilon_t\}_{t\geq1}$ are adapted to the filtration $\{\FF_t\}_{t\geq0}$, are bounded, and satisfy $\epsilon_t\to0$ $\prob$-a.s.
\end{enumerate}
\end{definition}

The following ODE convergence criterion follows from \cite[Theorem~2.1 and Section~2.2]{B2023}.

\begin{lemma}
\label{lem:sas}
Suppose that conditions~{\rm(1)--(5)} in Definition~\ref{def:sas} hold. Furthermore, suppose that the set of limit points of the solutions to the ODE $y'(u)=h(y(u))$, $u\geq0$, with $y(0)\in\R^d$ coincides with the set of zeroes of $h$, and that these zeroes are isolated. Then the iterates of the recursion in \eqref{eq:recursion} converge $\prob$-a.s. to a possibly sample-path-dependent zero of $h$.
\end{lemma}

\subsection{Martingale and fluctuation estimates}

\begin{theorem}[Duflo {\cite[Theorem~1.3.15]{Duflo1997}}]
\label{aux:duflo_martingale_estimate}
Let $(M_n)$ be a square-integrable martingale with increasing process $\langle M\rangle_n$, and set $\langle M\rangle_\infty:=\lim_{n\to\infty}\langle M\rangle_n$. Then:
\begin{itemize}
\item
On $\{\langle M\rangle_\infty<\infty\}$, $M_n\xrightarrow{\mathrm{a.s.}}M_\infty$, where $M_\infty$ is finite.

\item
On $\{\langle M\rangle_\infty=\infty\}$,
$M_n/\langle M\rangle_n\xrightarrow{\mathrm{a.s.}}0$. More precisely, for every $\gamma>0$, on $\{\langle M\rangle_\infty=\infty\}$,
\[
M_n
=
o\!\left(
\langle M\rangle_n^{1/2}
(\log\langle M\rangle_n)^{(1+\gamma)/2}
\right)
\qquad\text{a.s.}
\]
\end{itemize}
\end{theorem}

\begin{theorem}[Mokkadem--Pelletier {\cite[Theorem~1]{MP2005}}]
\label{aux:MP_compact_LIL}
Consider the stochastic approximation algorithm
\[
Z_{n+1}
=
Z_n+\gamma_n\bigl[h(Z_n)+r_{n+1}\bigr]
+\sigma_n\xi_{n+1},
\]
where $Z_n\in\IR^d$, $\gamma_n>0$, and $\sigma_n>0$. Set $v_n:=\gamma_n/\sigma_n^2$ and $s_n:=\sum_{k=1}^{n}\gamma_k$. Assume the following conditions.

\begin{enumerate}
\item[{\bf (1)}]
There exists a deterministic point $z_\star\in\IR^d$ such that $Z_n\to z_\star$ almost surely.

\item[{\bf (2)}]
The function $h$ is differentiable at $z_\star$, and its Jacobian matrix $H:=Dh(z_\star)$ satisfies
\[
\max\{\mathrm{Re}(\mu):\mu\in\mathrm{Sp}(H)\}=-L<0.
\]

\item[{\bf (3)}]
We have $\sigma_n\to0$, $\gamma_n\to0$, $s_n\to\infty$, $\gamma_n/\gamma_{n+1}=1+O(\gamma_n)$, and there exists $\rho_{\mathrm{MP}}\in[0,2L)$ such that
\[
\lim_{n\to\infty}
\frac1{\gamma_n}
\left(1-\frac{v_{n-1}}{v_n}\right)
=
\rho_{\mathrm{MP}}.
\]

\item[{\bf (4)}]
There exist a positive definite matrix $\Gamma$ and independent $\mathcal N(0,\Gamma)$ random vectors $(\eta_k)_{k\geq1}$ such that
\[
\left\|
\sum_{k=1}^{n}\xi_{k+1}
-
\sum_{k=1}^{n}\eta_{k+1}
\right\|_2
=
o\!\left(
\sqrt{\frac{\log s_n}{\gamma_n}}
\right)
\qquad\text{a.s.}
\]

\item[{\bf (5)}]
The remainder satisfies $r_{n+1}=r_{n+1}^{(1)}+o(\|Z_n-z_\star\|_2)$ a.s. with $\|r_{n+1}^{(1)}\|_2=O(v_n^{-1/2})$ a.s.
\end{enumerate}

Define
\[
\Sigma
:=
\int_0^\infty
\exp\!\left((H+\tfrac{\rho_{\mathrm{MP}}}{2}I)u\right)^\top
\Gamma
\exp\!\left((H+\tfrac{\rho_{\mathrm{MP}}}{2}I)u\right)
\,du.
\]
Then, with probability one, the sequence
$\left\{\sqrt{v_n/(2\log s_n)}\,(Z_n-z_\star)\right\}_{n\geq1}$
is relatively compact in $\IR^d$, and its set of limit points is the ellipsoid
$\{x\in\IR^d:x\Sigma^{-1}x^\top\leq1\}$.
\end{theorem}

\bibliographystyle{unsrt}
\bibliography{Bibliography_arxiv}

\end{document}